\documentclass[psamsfonts]{amsart}
\usepackage{amssymb}
\usepackage{graphicx,color}
\usepackage{amssymb,amsfonts}
\usepackage[all,arc]{xy}
\usepackage{enumerate}
\usepackage{mathrsfs}

\newtheorem{thm}{Theorem}[section]
\newtheorem{lemma}[thm]{Lemma}
\newtheorem{prop}[thm]{Proposition}
\newtheorem{cor}[thm]{Corollary}

\theoremstyle{definition}
\newtheorem{defn}[thm]{Definition}

\theoremstyle{remark}

\makeatletter
\let\c@equation\c@thm
\makeatother
\numberwithin{equation}{section}

\title[Hypersurfaces of Prescribed Curvature and Boundary ]{Strictly Locally Convex Radial Graphs of Prescribed Curvature and Boundary in Space Forms}

\author{Zhenan Sui}

\begin{document}

\begin{abstract}
We obtain $C^2$ a priori estimates for solutions of the nonlinear second-order elliptic equation related to the geometric problem of finding a strictly locally convex hypersurface with prescribed curvature and boundary in a space form. Under the assumption of a strictly locally convex subsolution, we establish existence results by using degree theory arguments.
\end{abstract}

\maketitle


\section {\large Introduction}

In this paper, we stay in $(n+1)$ dimensional space form $N^{n+1}(K)$ ($n \geq 2$) with constant sectional curvature $K = 0$, $1$ or $- 1$, which can be modeled as follows. In Euclidean space $\mathbb{R}^{n + 1}$, fix the origin $0$ and let $\mathbb{S}^n$ denote the unit sphere centered at $0$. Choose the spherical coordinates $(z, \rho)$ in $\mathbb{R}^{n + 1}$ with $z \in \mathbb{S}^n$. Define the new metric on $\mathbb{R}^{n+1}$ by
\[ \bar{g} = d \rho^2 + \phi^2(\rho)\, \sigma \]
where $\sigma$ is the standard metric on $\mathbb{S}^n$ induced from $\mathbb{R}^{n + 1}$. Then $(\mathbb{R}^{n + 1}, \bar{g})$ is a model of $N^{n + 1}(K)$ for $K = 0$ if we choose $\phi(\rho) = \rho$ where $\rho \in [\, 0, \infty)$, for $K = 1$ if $\phi(\rho) = \sin(\rho)$ where $\rho \in [\, 0, \pi/2)$, and for $K = - 1$ if $\phi(\rho) = \sinh(\rho)$ where $\rho \in [\, 0, \infty)$, which correspond to the Euclidean space $\mathbb{R}^{n + 1}$, the upper hemisphere $\mathbb{S}_{+}^{n + 1}$ and the hyperbolic space $\mathbb{H}^{n + 1}$ respectively. Let $V = \phi(\rho) \,\frac{\partial}{\partial \rho}$ be the conformal Killing field in $ N^{n+1}(K)$. It is well known that $V$ is the position vector field in Euclidean space.

Given a disjoint collection $\Gamma = \{ \Gamma_1, \ldots, \Gamma_m \}$
of closed smooth embedded $(n - 1)$ dimensional submanifolds, a smooth symmetric function $f$ of $n$ variables and a smooth positive function $\psi$ defined on $N^{n+1}(K)$, it is a fundamental question in differential geometry to seek a strictly locally convex hypersurface $\Sigma$ with the prescribed curvature
\begin{equation}  \label{eq1-0}
f (\kappa [ \Sigma ]) = \psi( V )
\end{equation}
and boundary
\begin{equation} \label{eq1-2}
\partial \Sigma = \Gamma
\end{equation}
where $\kappa [\Sigma] = (\kappa_1, \ldots, \kappa_n)$ denotes the principal curvatures of $\Sigma$ at $V$ with respect to the outward unit normal $\nu
$. We call a hypersurface $\Sigma$ strictly locally convex if all its principal curvatures $\kappa_i > 0$ everywhere in $\Sigma$.

Equation \eqref{eq1-0} arises in various geometric problems. If we do not impose boundary condition \eqref{eq1-2} and consider closed hypersurfaces, there is a vast literature in this direction. When requiring the convexity of the hypersurfaces, the Gauss curvature case was studied by Oliker \cite{Oliker} while the most current breakthrough is due to Guan-Ren-Wang \cite{GRW15}, where the authors studied convex hypersurfaces with prescribed Weingarten curvature in $\mathbb{R}^{n + 1}$ for general $\psi$ depending on both $V$ and $\nu$. For starshaped compact hypersurfaces, we refer the readers to \cite{BLO} for the introductory material, and see Jin-Li \cite{JL05} for Weingarten curvature in hyperbolic space, \cite{BLO, LO02} for Weingarten curvature in elliptic space, Spruck-Xiao \cite{SX15} for scalar curvature in space forms for general $\psi$, Chen-Li-Wang \cite{CLW18} for Weingarten curvature in warped product spaces for general $\psi$.

For the Dirichlet problem, important examples include the classical Plateau problem concerning the mean curvature as well as the corresponding problem for Gauss curvature (see \cite{CNSI, GS93, Guan95, Guan98, GS02}). The Dirichlet problem in the general setting \eqref{eq1-0}--\eqref{eq1-2} was first studied by Caffarelli-Nirenberg-Spruck \cite{CNSV} for vertical graphs over strictly convex domains in $\mathbb{R}^n$ with constant boundary data. Since then, there have been significant progresses, among which, we mention Guan-Spruck \cite{GS04} and Trudinger-Wang \cite{TW02} for general locally convex hypersurfaces in $\mathbb{R}^{n+1}$ which may not be graphs, Su \cite{Su16} for strictly locally convex radial graphs in $\mathbb{R}^{n+1}$ and Cruz \cite{Cruz} for starshaped radial graphs with prescribed Weingarten curvature in $\mathbb{R}^{n+1}$.

As in \cite{GS04}, the curvature function $f$ is assumed to be defined on the open symmetric convex cone $\Gamma_n^+ \equiv \{ \lambda \in \mathbb{R}^n \vert \, \lambda_i > 0, i = 1, \ldots, n \}$ satisfying the fundamental structure conditions
\begin{equation} \label{eq1-5}
f_i (\lambda) \equiv \frac{\partial f(\lambda)}{\partial \lambda_i} > 0 \quad \mbox{in} \quad  \Gamma_n^+,\quad i = 1, \ldots, n
\end{equation}
\begin{equation} \label{eq1-6}
f \,\,\mbox{is} \,\,\mbox{concave}\,\, \mbox{in} \,\, \Gamma_n^+
\end{equation}
\begin{equation} \label{eq1-7}
f > 0 \quad \mbox{in} \quad \Gamma_n^+, \quad f = 0 \quad \mbox{on} \quad \partial\Gamma_n^+
\end{equation}
In addition, $f$ is assumed to satisfy the technical conditions
\begin{equation} \label{eq1-10}
\sum f_i (\lambda) \lambda_i \geq \sigma_0 \quad \mbox{on} \quad \{ \lambda \in \Gamma_n^+ \vert \, \psi_0 \leq f(\lambda) \leq \psi_1 \}
\end{equation}
for any $\psi_1 > \psi_0 > 0$, where $\sigma_0$ is a positive constant depending only on $\psi_0$ and $\psi_1$, and for any $C > 0$ and any compact set $E \subset \Gamma_n^+$ there exists $R = R( E, C ) > 0$ such that
\begin{equation} \label{eq1-11}
f(\lambda_1, \ldots, \lambda_{n - 1}, \lambda_n + R) \geq C \quad \forall \,\,\lambda \in E
\end{equation}
Examples satisfying \eqref{eq1-5}--\eqref{eq1-11} include a large family $f = \sum f_l$ where
\[ f_l = S_n^{\frac{1}{n N_l}}  \prod\limits_{i = 1}^{N_l - 1} \Big( c_i + \sum\limits_{k = 1}^{n - 1} c_{i, k} S_{n, k}^{\frac{1}{n - k}}\Big)^{\frac{1}{N_l}} \]
where $c_i$, $c_{i, k} \geq 0$ are constants, $c_i + \sum_k c_{i, k} > 0$ for each $i$, $S_k$ is the $k$th elementary symmetric function, $S_0 = 1$ and $S_{k, l} = S_k / S_l \,(0 \leq l < k \leq n)$. However, the pure curvature quotient $S_{n, k}^{1 / (n - k)}$ does not satisfy \eqref{eq1-11}.

In this paper, we are interested in strictly locally convex hypersurfaces embedded in $N^{n+1}(K)$ which can be represented as radial graphs over a domain in $\mathbb{S}^n$. Assuming $\Gamma$ to be the boundary of a smooth positive radial graph $\varphi$ in $N^{n + 1}(K)$ defined on a smooth domain $\Omega \subset \mathbb{S}^n$, we thus have $\Gamma = \{ (z, \varphi(z))\, | z \in \partial \Omega \}$ and look for a smooth strictly locally convex radial graph $\Sigma = \{ (z, \rho(z))\, | z \in \Omega \}$ satisfying the Dirichlet problem
\begin{equation} \label{eq1-1}
f(\kappa [ \rho ]) = \psi(z, \rho) \quad \mbox{in} \quad \Omega
\end{equation}
\begin{equation} \label{eq1-3}
\rho = \varphi \quad \mbox{on} \quad \partial \Omega
\end{equation}
where $\kappa[\rho]$ denotes the principal curvatures of the graph of $\rho$ and we use the same $\psi$ for the smooth positive function on the right hand side.
For $C^0$ estimates, we assume that
\begin{equation} \label{eq1-9}
\Omega \,\,\, \mbox{does}\,\,\mbox{not}\,\,\mbox{contain}\,\,\mbox{any}\,\,\mbox{hemisphere}.
\end{equation}

We obtain the following $C^2$ estimates:
\begin{thm} \label{Theorem1-1}
Under assumption \eqref{eq1-5}--\eqref{eq1-11} and \eqref{eq1-9}, suppose $\Gamma$  can span  a $C^2$ positive radial graph $\overline{\rho}$ in $N^{n + 1}(K)$ which is strictly locally convex in a neighborhood of $\Gamma$. Then for any $C^4$ strictly locally convex radial graph $\rho$ satisfying \eqref{eq1-1}-\eqref{eq1-3}  with $\rho \leq \overline{\rho}$ in $\Omega$, we have
\[ \Vert \rho \Vert_{C^2(\overline{\Omega})} \,\leq \, C \]
where $C$ depends only on $\Omega$, $\Vert\psi\Vert_{C^2}$, $\Vert \overline{\rho}\Vert_{C^1(\overline{\Omega})}$,  $\Vert \varphi \Vert_{C^4(\overline{\Omega})}$, $\inf \psi$, $\inf_{\partial\Omega}\overline{\rho}$ and the convexity of $\overline{\rho}$.
\end{thm}

We remark that for $C^2$ estimates, it is necessary in Theorem \ref{Theorem1-1} to assume $\overline{\rho}$ to be strictly locally convex near its boundary. To establish existence results, as in \cite{GS93, Guan95, Guan98, GS02, GS04, Su16}, we further require that $\overline{\rho}$ is a strictly locally convex subsolution. Since there are topological obstructions to the existence of strictly locally convex hypersurfaces spanning a given $\Gamma$ (see \cite{Ro93}), the existence of a subsolution allows the arbitrary geometry of $\Gamma$.
Using Theorem \ref{Theorem1-1}, we can prove the following existence results.

\begin{thm} \label{Theorem1-2}
Under assumption \eqref{eq1-5}--\eqref{eq1-11} and \eqref{eq1-9}, assume in addition that there exists a smooth strictly locally convex radial graph $\overline{\rho}$ satisfying
\begin{equation} \label{eq1-4}
\begin{aligned}
f (\kappa [ \overline{\rho} ])\, \geq & \,\, \psi(z, \overline{\rho}) \quad & \mbox{in} \quad &\Omega \\
\overline{\rho}\, = & \,\, \varphi \quad &\mbox{on} \quad &\partial \Omega
\end{aligned}
\end{equation}
Then there exists a smooth strictly locally convex radial graph $\Sigma = \{ (z, \rho(z))\,\vert\, z \in \Omega \}$ in space form $N^{n+1}(K)$  satisfying the Dirichlet problem \eqref{eq1-1}-\eqref{eq1-3} with $\rho \leq \overline{\rho}$ in $\overline{\Omega}$ and uniformly bounded principal curvatures
\[ 0 < K_0^{-1} \leq \kappa_i \leq K_0  \quad \mbox{on} \quad \Sigma. \]
where $K_0$ is a uniform positive constant depending only on $\Omega$, $\Vert\psi\Vert_{C^2}$, $\Vert \overline{\rho}\Vert_{C^1(\overline{\Omega})}$,  $\Vert \varphi \Vert_{C^4(\overline{\Omega})}$, $\inf \psi$, $\inf_{\partial\Omega}\overline{\rho}$ and the convexity of $\overline{\rho}$.
\end{thm}

In Euclidean space $\mathbb{R}^{n+1}$, Theorem \ref{Theorem1-2} was proved in \cite{GS93} for constant Gauss curvature assuming the existence of a strictly locally convex strict subsolution and was extended in \cite{Guan95} for general $\psi$ depending also on the gradient term.
These existence results are established via the theory of Monge-Amp\`ere type equations on $\mathbb{S}^n$. The linearized operators may have nontrivial kernels, which call for extra efforts for the proof of existence since one can not directly use continuity method. In \cite{GS93}, the authors established the existence results for equations with ${\partial\psi} / {\partial u} \leq 0$ by monotone iteration approach. In \cite{Guan95} the author rederived $C^2$ estimates for a wider class of equations which allows the application of degree theory to the proof of existence for general $\psi$ (the proof also need the existence result in \cite{GS93}). In \cite{Guan98}, Guan obtained the existence results for Monge-Amp\`ere equations with general $\psi$ over smooth bounded domains in $\mathbb{R}^{n}$ by assuming the existence of a subsolution (improving the results in \cite{CNSI} where the authors assumed the strict convexity of the domain) and stated that the strict subsolution assumption in \cite{GS93, Guan95} can be weakened to a subsolution. More recently, Su \cite{Su16} proved Theorem \ref{Theorem1-2} in $\mathbb{R}^{n + 1}$ assuming the existence of a strict subsolution, where the author reformulated \eqref{eq1-1} in a form with invertible linearized operator and thus continuity method and degree theory can be directly applied without extra $C^2$ estimates.

The novelty of this paper lies in: first, it provides a unified approach for $C^2$ estimates by transformation (see \eqref{eq3-14}). Second, for proving existence by degree theory, it generalizes Su's idea ( see \cite{Su16} ) to $\mathbb{H}^{n+1}$ and weaken the strict subsolution assumption. Besides, it creates a new continuity process starting from $\mathbb{R}^{n+1}$ to $\mathbb{S}^{n+1}_+$ and hence the existence in $\mathbb{S}^{n+1}_+$ is proved.

This paper is organized as follows: in section 2, we reformulate equation \eqref{eq1-1} in two different ways: one is used for deriving $C^2$ boundary estimates in section 3 and the other is for proving existence in $\mathbb{R}^{n+1}$ and $\mathbb{H}^{n+1}$ in section 5. Section 4 is devoted to global $C^2$ estimates. Section 6 is for existence in $\mathbb{S}^{n+1}_+$.

\vspace{5mm}

\section{Strictly locally convex radial graphs in space forms and reformulations of equation \eqref{eq1-1}}

\vspace{3mm}

Throughout this paper, we are interested in hypersurface $\Sigma \subset N^{n + 1} (K)$ that can be represented as a smooth radial graph over a smooth domain $\Omega \subset \mathbb{S}^n$, i.e.
\[ \Sigma  = \{ (z, \rho(z))\,| z \in \Omega \} \]
We note that the range for $\rho = \rho(z)$ is $ (0, \rho_U^K)$ where
\begin{equation} \label{eq3-15}
\rho_U^K = \left\{ \begin{aligned} & \infty,\quad\quad\mbox{if} \quad K = 0 \quad \mbox{or}\quad - 1 \\
& \frac{\pi}{2},\quad\quad\mbox{if} \quad K = 1
\end{aligned}
\right.
\end{equation}

Following the notations in \cite{SX15}, we introduce the following geometric quantities on $\Sigma$. Let $\nabla'$ denote the covariant derivatives with respect to some local orthonormal frame $e_1, \ldots, e_n$ on $\mathbb{S}^n$ (while $\nabla$ will be reserved for the covariant derivatives with respect to some local orthonormal frame $E_1, \ldots, E_n$ on $\Sigma$). The induced metric, its inverse, unit outer normal, and second fundamental form on $\Sigma$ are given respectively by
\begin{equation} \label{eq3-1}
 g_{ij} = \phi^2\, \delta_{ij} + \rho_i \rho_j
\end{equation}
\begin{equation} \label{eq3-2}
 g^{ij} = \frac{1}{\phi^2} \big( \delta_{ij} - \frac{\rho_i \rho_j}{\phi^2 + |\nabla'\rho|^2} \big)
\end{equation}
\begin{equation} \label{eq3-4}
 \nu = \frac{- \nabla' \rho + \phi^2 \,\frac{\partial}{\partial \rho}}{\sqrt{\phi^4 + \phi^2 |\nabla' \rho|^2}}
\end{equation}
\begin{equation} \label{eq3-5}
 h_{ij} = \frac{\phi}{\sqrt{\phi^2 + |\nabla' \rho|^2}} \,\big( - \nabla'_{ij} \rho + \frac{2 \phi'}{\phi}\,\rho_i \rho_j + \phi \phi' \delta_{ij} \big)
\end{equation}
where $\rho_i = \rho_{e_i} = \nabla'_{e_i} \rho = \nabla'_{i} \rho $, $\rho_{ij} = \nabla'_{e_j} \nabla'_{e_i} \rho =  \nabla'_{e_j e_i} \rho = \nabla'_{ji} \rho = \nabla'_{ij} \rho$, and higher order covariant derivatives are interpreted in this manner. We thus have $\nabla'\rho = \rho_k \, e_k$ (while in Section 4, $\rho_i$ may denote $\nabla_{E_i} \rho$, which is the covariant derivative with respect to $E_1, \ldots, E_n$).

The principal curvatures $\kappa_1, \ldots, \kappa_n$ of the radial graph $\rho$ are the eigenvalues of the real symmetric matrix $ \{ a_{ij} \}$:
\[ a_{ij} = \gamma^{i k} \,h_{k l}\, \gamma^{l j} \]
with $\{ \gamma^{ik} \}$ and its inverse $\{ \gamma_{ik} \}$ given respectively by
\begin{equation} \label{eq3-6}
 \gamma^{ik} = \frac{1}{\phi} ( \delta_{ik} - \frac{\rho_i \,\rho_k}{\sqrt{\phi^2 + |\nabla' \rho|^2} ( \phi + \sqrt{\phi^2 + |\nabla' \rho|^2} )} )
\end{equation}
\begin{equation} \label{eq3-7}
\gamma_{ik} = \phi \,\delta_{ik} + \frac{\rho_i \rho_k}{ \phi + \sqrt{\phi^2 + |\nabla' \rho|^2} }
\end{equation}
Note that $\{ \gamma_{ik} \}$ is the square root of the metric, i.e., $\gamma_{ik} \gamma_{kj} = g_{ij}$.

\begin{defn}
A hypersurface $\Sigma$ is strictly locally convex if its principal curvatures are all positive, i.e.
$\kappa_i > 0$ for $i = 1, \ldots, n$ everywhere on $\Sigma$;
or, equivalently, the symmetric matrix $\{ a_{ij }\}$ (or $\{ h_{ij} \}$) is positive definite everywhere in $\Omega$.

A $C^2$ function $\rho$ is strictly locally convex if the hypersurface $\Sigma$ represented by $\rho$ is strictly locally convex.
\end{defn}
For simplicity, throughout this paper  $a_{ij} > 0 $ (or $\geq 0$ ) means that the symmetric matrix $\{ a_{ij} \}$ is positive definite (or positive semi-definite); and $a_{ij} \geq b_{ij}$ means that the symmetric matrices $\{a_{ij}\}$ and $\{b_{ij}\}$ satisfy $ a_{ij} - b_{ij} \geq 0 $.
Now we will transform $\rho$ into other variables for deriving a priori estimates and proving the existence.

\vspace{5mm}

\subsection{Reformulation for deriving a priori estimates}~

\vspace{3mm}

We do the following transformation
\begin{equation} \label{eq3-14}
\rho = \zeta ( u ) \, = \,\left\{ \begin{aligned} & \frac{1}{u},\quad\quad & \mbox{if} \quad  K = 0  \\
& \mbox{arccot}\, u,\quad\quad & \mbox{if} \quad  K = 1  \\
& \frac{1}{2} \ln \big( \frac{ u + 1 }{u - 1} \big) ,\quad\quad & \mbox{if} \quad  K = - 1
\end{aligned}
\right.
\end{equation}
In view of \eqref{eq3-15}, the range for $u$ is  $(u_L^K, \infty)$ with
\begin{equation} \label{eq3-16}
u_L^K = \left\{ \begin{aligned} & 0,\quad\quad\mbox{if} \quad K = 0 \quad \mbox{or}\quad 1 \\
& 1,\quad\quad\mbox{if} \quad K = -1
\end{aligned}
\right.
\end{equation}
Then the formula \eqref{eq3-1}, \eqref{eq3-2}, \eqref{eq3-6}, \eqref{eq3-7} and \eqref{eq3-5} can be expressed as
\begin{equation} \label{eq3-8}
g_{ij} = \phi^2 \,\delta_{ij} + \zeta'^2(u)\, u_i u_j
\end{equation}

\begin{equation} \label{eq3-9}
g^{ij} = \frac{1}{\phi^2} \Big( \delta_{ij} - \frac{\zeta'^2(u) u_i u_j}{\phi^2 + \zeta'^2(u) |\nabla' u|^2} \Big)
\end{equation}

\begin{equation} \label{eq3-10}
\gamma^{ik} =  \,\frac{1}{\phi}\, \Big( \delta_{ik} - \frac{ \zeta'^2(u) u_i u_k }{\sqrt{\phi^2 + \zeta'^2(u) |\nabla' u|^2 } ( \phi + \sqrt{\phi^2 + \zeta'^2(u) |\nabla' u|^2 } )}\Big)
\end{equation}

\begin{equation} \label{eq3-11}
 \gamma_{ik} =  \phi \,\delta_{ik} + \frac{\zeta'^2(u) u_i u_k}{ \phi + \sqrt{\phi^2 + \zeta'^2(u) |\nabla' u|^2 } }
\end{equation}

\begin{equation} \label{eq3-13}
h_{ij} =  \frac{- \zeta'(u) \phi}{\sqrt{\phi^2 + \zeta'^2 |\nabla' u|^2}} ( \nabla'_{ij} u + u \,\delta_{ij} )
\end{equation}
Hence
\begin{equation} \label{eq3-18}
a_{ij} = \frac{- \zeta'(u) \phi}{\sqrt{\phi^2 + \zeta'^2 |\nabla' u|^2}} \gamma^{ik}\, ( \nabla'_{kl} u + u \,\delta_{kl} )\,\gamma^{lj}
\end{equation}
It is easy to see that  $\Sigma$ (or $u$) is strictly locally convex if and only if
\begin{equation} \label{eq3-17}
 \nabla'_{ij} u + \, u \,\delta_{ij}  > 0 \quad\mbox{in} \quad \Omega
\end{equation}

Under transformation \eqref{eq3-14},
the Dirichlet problem \eqref{eq1-1}-\eqref{eq1-3}
is equivalent to
\begin{equation} \label{eq2-11}
f(\kappa[ u ]) = \, \psi(z, u)  \quad  \mbox{in} \quad \Omega
\end{equation}
\begin{equation} \label{eq2-12}
\quad \,\, u =  \,\varphi  \quad \quad \mbox{on} \quad \partial \Omega
\end{equation}
Here we still use $\psi$ for the function on the right hand side, and $\varphi$ for the boundary value.
Denote $\kappa[u] = (\kappa_1, \ldots, \kappa_n) = \lambda( A [ u ] )$ where $\lambda(A)$ denotes the eigenvalues of $A$ and $A[u] = \{ a_{ij} \}$ with $a_{ij}$ given by \eqref{eq3-18}.
Define the function $F$ by $ F( A ) = f(\lambda( A ))$ and the function $G$ by
\[ G(r, p, u) \,= \, F ( A ( r, p, u ) )\]
where $A( r, p, u )$ is obtained from $A[u]$ with $(\nabla'^2 u, \nabla' u, u)$ replaced by $( r, p, u )$. Therefore
equation \eqref{eq2-11} can be rewritten in the following form
\begin{equation} \label{eq2-13}
G(\nabla'^2 u, \nabla' u, u) =  \,\psi(z, u)  \quad  \mbox{in} \quad \Omega
\end{equation}
Denote
\[ F^{ij} ( A ) = \frac{\partial F}{\partial a_{ij}} ( A ), \quad F^{ij, kl} (A) = \frac{\partial^2 F}{\partial a_{ij} \partial a_{kl}} (A)\]
\[ G^{ij}(r, p, u) = \frac{\partial G}{\partial r_{ij}}(r, p, u), \quad G^i(r, p, u) = \, \frac{\partial G}{\partial p_i}(r, p, u),  \quad G_u(r, p, u) = \, \frac{\partial G}{\partial u}(r, p, u)\]
\[ \psi_u(z, u) = \, \frac{\partial \psi}{\partial u}(z, u) \]
As mentioned in \cite{GS04}, the function $F$ possesses the following properties.
First, the matrix $\{ F^{ij} (A) \}$ is symmetric with eigenvalues $f_1, \ldots, f_n$.
By \eqref{eq1-5}, $F^{ij} (A) > 0$ whenever $\lambda(A) \in \Gamma_n^+$, and by \eqref{eq1-6} we know that $F$ is a concave function of $A$, i.e., the symmetric matrix $F^{ij, kl} (A) \leq 0 $ whenever $\lambda(A) \in \Gamma_n^+$.
The function $G$ satisfies similar structure conditions as $F$. In fact, from \eqref{eq3-18} we have
\begin{equation} \label{eq2-14}
 G^{ij} = \frac{\partial G}{\partial u_{ij}}  = \frac{\partial F}{\partial a_{kl}} \frac{\partial a_{kl}}{\partial u_{ij}} =  \frac{ - \phi \zeta'(u)}{\sqrt{\phi^2 + \zeta'^2(u) |\nabla' u|^2}} F^{kl} \gamma^{ik} \gamma^{jl}
\end{equation}
Thus the symmetric matrix $G^{ij} > 0$ if and only if $F^{ij} > 0$, which in particular implies that equation \eqref{eq2-13} is elliptic for strictly locally convex solutions.
Also by \eqref{eq3-18} we can calculate
\[ \frac{\partial^2 G}{\partial u_{ij} \partial u_{kl}} = \frac{\partial a_{pq}}{\partial u_{ij}} \,\frac{\partial^2 F}{\partial a_{pq} \partial a_{rs}}\,\frac{\partial a_{rs}}{\partial u_{kl}}\]
which implies that $G$ is concave with respect to $\{ u_{ij} \}$ for strictly locally convex $u$.

We next compute $G^s$ and $G_u$, which will be needed in section 3.
\begin{lemma} \label{Lemma2-1}
Denote $w = \sqrt{ \phi^2  + \zeta'^2(u) |\nabla' u|^2}$. Then
\begin{equation} \label{eq2-15}
 G^s \, = \, - \frac{2 \zeta'^2 ( w \,\gamma^{is}\, u_q + \phi \,\gamma^{q s} u_i )}{ w ( \phi + w )} F^{ij} a_{q j} - \frac{\zeta'^2\, u_s}{w^2} \, F^{ij} a_{ij}
 \end{equation}
\begin{equation} \label{eq2-31}
\begin{aligned}
G_u = -2 \Big( \phi \phi' \zeta' g^{iq} + \frac{\zeta' \zeta'' u_i u_q}{w^2}\Big) F^{ij} a_{q j} + \Big( \frac{\phi' \zeta'}{\phi} - \frac{\phi \phi' \zeta'}{w^2} + \frac{\phi^2 \zeta''}{\zeta'\, w^2} \Big) F^{ij} a_{ij} - \frac{\phi \zeta'}{w} F^{ij} g^{ij}
\end{aligned}
\end{equation}
\end{lemma}

\begin{proof}
We first prove \eqref{eq2-15}. Note that
\begin{equation}  \label{eq2-16}
G^s \, = \frac{\partial F}{\partial a_{ij}} \frac{\partial a_{ij}}{\partial u_s} = F^{ij} \big( 2\, \frac{\partial \gamma^{ik}}{\partial u_s} \, h_{kl}\,\gamma^{lj} + \gamma^{ik}\,\frac{\partial h_{kl}}{\partial u_s}\,\gamma^{lj} \big)
\end{equation}
where
\begin{equation} \label{eq2-17}
\frac{ \partial \gamma^{ik}}{\partial u_s} = - \gamma^{ip}\, \frac{\partial \gamma_{pq}}{\partial u_s} \, \gamma^{qk}
\end{equation}
Direct calculations from \eqref{eq3-11} and \eqref{eq3-10} yield
\begin{equation} \label{eq2-18}
\frac{\partial \gamma_{pq}}{\partial u_s} \,=
   \frac{ \zeta'^2(u) ( \delta_{ps} u_q + \delta_{q s} u_p ) }{ \phi + w } - \frac{\zeta'^4( u )\, u_p u_q u_s}{( \phi + w )^2 w}
= \frac{\zeta'^2(u) ( \delta_{ps} u_q + \phi\, u_p \gamma^{q s}) }{ \phi + w  }
\end{equation}
and
\begin{equation} \label{eq2-19}
\gamma^{ip} \,u_p =  \frac{ u_i }{w}
\end{equation}
Besides, from \eqref{eq3-13} and \eqref{eq3-18} we have
\begin{equation} \label{eq2-20}
\gamma^{ik}\,\frac{\partial h_{kl}}{\partial u_s}\,\gamma^{lj} = - \frac{\zeta'^2(u) \,u_s}{ w^2 } \, a_{ij}
\end{equation}
Taking \eqref{eq2-17}--\eqref{eq2-20} into \eqref{eq2-16}, the formula \eqref{eq2-15} is proved.

The formula \eqref{eq2-31} can be proved similarly. In fact,
\begin{equation}  \label{eq2-36}
G_u \, = \frac{\partial F}{\partial a_{ij}} \frac{\partial a_{ij}}{\partial u} = F^{ij} \big( 2\, \frac{\partial \gamma^{ik}}{\partial u} \, h_{kl}\,\gamma^{lj} + \gamma^{ik}\,\frac{\partial h_{kl}}{\partial u}\,\gamma^{lj} \big)
\end{equation}
where
\begin{equation*}
\frac{ \partial \gamma^{ik}}{\partial u} = - \gamma^{ip}\, \frac{\partial \gamma_{pq}}{\partial u} \, \gamma^{qk}
\end{equation*}
From \eqref{eq3-11} we have
\begin{equation*}
\begin{aligned}
 \frac{\partial \gamma_{ik}}{\partial u} = & \, \phi' \zeta' \delta_{ik} + \frac{2 \zeta' \zeta'' u_i u_k}{\phi + w } - \frac{\zeta'^2(u) u_i u_k}{( \phi + w )^2} \big( \phi' \zeta'(u) + \frac{\phi \phi' \zeta' + \zeta'\zeta''|\nabla' u|^2}{w} \big)
 \\
= & \,\phi' \zeta' \delta_{ik} + \frac{ \zeta' u_i u_k}{\phi + w} \big( 2 \zeta'' - \frac{\zeta'}{\phi + w} \big( \phi' \zeta'  +  \frac{\phi \phi' \zeta' + \zeta'\zeta''|\nabla' u|^2}{w}   \big)\, \big) \\
= & \,\phi' \zeta' \delta_{ik} + \frac{ \zeta' u_i u_k}{\phi + w} \big( 2 \zeta'' - \frac{\zeta'^2 \zeta'' |\nabla' u|^2}{(\phi + w) w} - \frac{\phi' \zeta'^2}{w} \big) \\
= & \,\phi' \zeta' \delta_{ik} + \frac{ \zeta' u_i u_k}{\phi + w} \big(  \frac{w + \phi}{w} \zeta''  - \frac{\phi' \zeta'^2}{w} \big)
\end{aligned}
\end{equation*}
In view of \eqref{eq3-10}, the above formula becomes
\begin{equation} \label{eq2-37}
 \frac{\partial \gamma_{ik}}{\partial u} = \phi \phi' \zeta' \gamma^{ik} + \frac{\zeta' \zeta'' u_i u_k}{w}
\end{equation}
Direct calculation from \eqref{eq3-13} yields
\begin{equation} \label{eq2-38}
\frac{\partial h_{ij}}{\partial u} = ( - \frac{\phi' \zeta'^2}{w} + \frac{\phi^2 \phi' \zeta'^2}{w^3} - \frac{\phi^3 \zeta''}{w^3} ) (\nabla'_{ij} u + u \delta_{ij}) - \frac{\phi\, \zeta'}{w} \delta_{ij}
\end{equation}
Inserting \eqref{eq2-37} and \eqref{eq2-38} into \eqref{eq2-36} and in view of \eqref{eq3-18} and \eqref{eq2-19} we obtain
\eqref{eq2-31}.
\end{proof}
\begin{cor} \label{GsGu}
Suppose that we have the $C^1$ bounds for strictly locally convex solutions $u$ of \eqref{eq2-11}:
\begin{equation*}
 u_L^K < C_0^{-1} \leq u \leq C_0, \quad\quad |\nabla' u| \leq C_1 \quad\mbox{in} \quad \overline{\Omega}
\end{equation*}
Then
\[| G^s | \, \leq \, C \quad \mbox{and} \quad
| G_u | \,\leq \, C ( 1 + \sum G^{ii} )
\]
\end{cor}
\begin{proof}
Note that $\{ F^{ij} (A) \}$ and $A$ can be diagonalized simultaneously by an orthonormal transformation. Consequently, the eigenvalues of the matrix $\{ F^{ij}(A)\} A$, which is not necessarily symmetric, are given by
\[ \lambda( \{ F^{ij}(A)\} A ) = ( f_1 \kappa_1, \ldots, f_n \kappa_n ) \]
In particular we have
\[ F^{ij}\, a_{ij} = \sum f_i \kappa_i \]
In addition,
for a bounded matrix $B = \{b_{ij}\}$, i.e. $|b_{ij}| \leq C$ for all $1 \leq i, j \leq n$ we have
\[ \vert b_{i k} F^{i j } a_{k j} \vert \, \leq C \sum f_i \kappa_i \]
Thus from \eqref{eq2-15} and \eqref{eq2-31} we have
\[ \vert G^s \vert \,\leq \, C  \sum f_i \kappa_i\quad \mbox{and} \quad \vert G_u \vert \,\leq \,C ( \sum f_i \kappa_i + \sum f_i ) \]
Finally, by the concavity of $f$ and $f(0) = 0$ we can derive that $ \sum f_i \kappa_i \leq \psi \leq C$. Also, in view of \eqref{eq2-14} we have
$\sum f_i \leq C \sum G^{ii}$. Hence the corollary is proved.
\end{proof}

\vspace{5mm}

\subsection{ Reformulation for proving existence }~

\vspace{3mm}

For the proof of the existence in Section 5, we do the following transformation.
\begin{equation} \label{eq6-2}
u = \eta ( v ) \, = \,\left\{ \begin{aligned} & e^v,\quad\quad & \mbox{if} \quad  K = 0  \\
& \sinh v,\quad\quad & \mbox{if} \quad  K = 1  \\
& \cosh v,\quad\quad & \mbox{if} \quad  K = - 1
\end{aligned}
\right.
\end{equation}
In view of \eqref{eq3-16}, the range for $v$ is $(v_L^K, \infty)$ with
\begin{equation} \label{eq6-19}
v_L^K = \left\{ \begin{aligned} & - \infty,\quad\quad &\mbox{if} \quad K = 0  \\
& 0,\quad\quad &\mbox{if} \quad K = 1 \,\, \mbox{or}\,\, -1
\end{aligned}
\right.
\end{equation}
The formula \eqref{eq3-10} and \eqref{eq3-13} can consequently be transformed into
\begin{equation} \label{eq6-3}
\gamma^{ik} =  \,\eta'(v) \,\Big( \delta_{ik} - \frac{ v_i v_k }{\sqrt{1 +  |\nabla' v|^2 } ( 1 + \sqrt{ 1 +  |\nabla' v|^2 } )} \Big)
\end{equation}
\begin{equation} \label{eq6-4}
h_{ij} = \, \frac{1}{ \eta'^2(v) \sqrt{1 + |\nabla' v|^2}} \Big( \eta'(v) \nabla'_{ij} v + \eta(v) v_i v_j +  \eta(v) \,\delta_{ij} \Big)
\end{equation}
Denoting
\[ w = \sqrt{1 + |\nabla' v|^2}\quad \mbox{and} \quad \tilde{\gamma}^{ik} = \delta_{ik} - \frac{ v_i v_k }{w ( 1 + w )}\]
we therefore have
\begin{equation} \label{eq6-1}
\begin{aligned}
a_{ij} =  & \, \frac{1}{ w } \, \tilde{\gamma}^{ik} \Big( \eta'(v) \nabla'_{kl} v + \eta(v) v_k v_l +  \eta(v) \,\delta_{kl}\Big)\, \tilde{\gamma}^{lj} \\
= & \, \frac{1}{w} \Big(\, \eta(v)\, \delta_{ij} \,+ \,\eta'(v)\, \tilde{\gamma}^{ik}\, \nabla'_{kl} v  \,\tilde{\gamma}^{lj} \,\Big)
\end{aligned}
\end{equation}
It is easy to see that $\Sigma$ (or $v$) is strictly locally convex if and only if
\begin{equation} \label{eq6-6}
 \eta(v)\, \delta_{ij} \,+ \,\eta'(v)\, \tilde{\gamma}^{ik} \,\nabla'_{kl} v \, \tilde{\gamma}^{lj}  > 0 \quad \mbox{in} \quad \Omega
\end{equation}

Under transformation \eqref{eq6-2}, we now reformulate \eqref{eq2-11}--\eqref{eq2-12} in terms of $v$,
\begin{equation} \label{eq2-24}
f(\kappa[ v ]) = \, \psi(z, v)  \quad  \mbox{in} \quad \Omega
\end{equation}
\begin{equation} \label{eq2-25}
v =  \,\varphi  \quad \quad \mbox{on} \quad \partial \Omega
\end{equation}
Here we still use $\psi$ for the right function and $\varphi$ for the boundary value. At this time,
$\kappa[v] = (\kappa_1, \ldots, \kappa_n) = \lambda( A [ v ] )$ and $A[v] = \{ a_{ij} \}$ with $a_{ij}$ given by \eqref{eq6-1}. Replacing $(\nabla'^2 v, \nabla' v, v)$ in $A[v]$ by $( r, p, v )$ we can define $A( r, p, v )$. Then we can
define $\mathcal{G}$ by
$\mathcal{G}(r, p, v) \,= \, F ( A ( r, p, v ) )$.
Thus, equation \eqref{eq2-24} can be rewritten as
\begin{equation} \label{eq2-26}
\mathcal{G}(\nabla'^2 v, \nabla' v, v) =  \,\psi(z, v)  \quad  \mbox{in} \quad \Omega
\end{equation}
Denote
\[ \mathcal{G}^{ij}(r, p, v) = \frac{\partial \mathcal{G}}{\partial r_{ij}} (r, p, v), \quad \mathcal{G}^{i} (r, p, v) = \frac{\partial \mathcal{G}}{\partial p_{i}}(r, p, v), \quad \mathcal{G}_v (r, p, v) = \, \frac{\partial \mathcal{G}}{\partial v} (r, p, v) \]
By \eqref{eq6-1}, we notice that equation \eqref{eq2-26} is elliptic for strictly locally convex $v$, and $\mathcal{G}$ is concave with respect to $\nabla'^2 v $ for strictly locally convex $v$.

\vspace{5mm}

\section{ Second order boundary estimates}

\vspace{5mm}

In this section, we derive the $C^2$ a priori estimates for strictly locally convex solutions $u$ to the Dirichlet problem \eqref{eq2-13}-\eqref{eq2-12} with $u \geq \underline{u}$ in $\Omega$
\begin{equation} \label{eq4-2}
\Vert u \Vert_{C^{2}(\overline{\Omega})} \,\leq\, C
\end{equation}
which, together with $u \geq C_0^{-1} > u_L^K$, implies an upper bound for all the principal curvatures by \eqref{eq3-18}. By assumption \eqref{eq1-7}, the principal curvatures admit a uniform positive lower bound, i.e.,
\begin{equation} \label{eq4-3}
0 < K_0^{-1} \leq \kappa_i \leq K_0 \quad \mbox{in} \quad \Omega
\end{equation}
which implies the uniform ellipticity of the linearized operator. Then $C^{2, \alpha}$ estimates can be established by Evans-Krylov theory \cite{Evans,Krylov}
\begin{equation} \label{eq4-1}
\Vert u \Vert_{C^{2,\alpha}(\overline{\Omega})} \,\leq\, C
\end{equation}
and higher-order regularity follows from classical Schauder theory.

The following $C^1$ estimates have been established in \cite{GS93} which was originally stated for $\mathbb{R}^{n + 1}$, but it also works in space forms.
\begin{lemma} \label{Lemma5-1}
Under assumption \eqref{eq1-9}, for any strictly locally convex function $u$ with $ u \geq \underline{u}$ in $\Omega$ and $u = \underline{u}$ on $\partial\Omega$ we have
\begin{equation} \label{eq5-1}
 u_L^K < C_0^{-1} \leq u \leq C_0, \quad\quad |\nabla' u| \leq C_1 \quad\mbox{in} \quad \overline{\Omega}
\end{equation}
where $C_0$ depends only on $\Omega$, $\sup_{\partial\Omega}\underline{u}$ and $\inf_{\Omega}\underline{u}$; $C_1$ depends in addition on $\sup_{\partial\Omega}\vert\nabla'\underline{u}\vert$.
\end{lemma}

In next section, we will derive global curvature estimates, which is equivalent to the global estimates for $|\nabla'^2 u|$ on $\overline{\Omega}$ from its bound on $\partial\Omega$. Therefore in this section we focus on the boundary estimates:
\begin{equation} \label{eq3-32}
| \nabla'^2 u | \leq C \quad\mbox{on} \quad \partial\Omega
\end{equation}

Consider any fixed point $z_0 \in \partial \Omega$. Choose a local orthonormal frame field $e_1, \ldots, e_n$ around $z_0$ on $\Omega$, which is obtained by parallel translation of a local orthonormal frame field on $\partial\Omega$ and the interior, unit, normal vector field to $\partial \Omega$, along the geodesics perpendicular to $\partial \Omega$ on $\Omega$. Assume that $e_n$ is the parallel translation of the unit normal field on $\partial \Omega$.

Since $u = \varphi$ on $\partial \Omega$,
\[ \nabla'_{\alpha\beta} (u - \varphi) = - \nabla'_n ( u - \varphi )\,\Gamma_{\alpha\beta}^n, \quad \alpha, \beta < n \quad\mbox{on}\quad \partial\Omega \]
where $\Gamma_{ij}^k$ are the Christoffel symbols of $\nabla'$ with respect to the frame $e_1, \ldots, e_n$ on $\mathbb{S}^n$. We thus obtain
\begin{equation} \label{eq3-27}
| \nabla'_{\alpha \beta} u (z_0) | \leq C, \quad \alpha, \beta < n
\end{equation}

Let $\rho(z)$ and $d(z)$ denote the distances from $z \in \overline{\Omega}$ to $z_0$ and $\partial \Omega$ on $\mathbb{S}^n$, respectively.
Set
\[ \Omega_\delta = \{ z \in \Omega : \rho(z) < \delta \}  \]
Choose $\delta_0 > 0$ sufficiently small such that $\rho$ and $d$ are smooth in $\Omega_{\delta_0}$, on which, we have
\[ |\nabla' d| = 1, \quad\quad - C\, I \leq \nabla'^2\, d \leq C \,I, \quad
\quad |\nabla' \rho| = 1, \quad\quad I \,\leq \,\nabla'^2 \,\rho^2 \,\leq \, 3 I \]
where $C$ depends only on $\delta_0$ and the geometric quantities of $\partial \Omega$, and
\[ \nabla'^2 \underline{u} + \underline{u} \, I \geq 4 \, c_0 \, I  \]
for some constant $c_0 > 0$ in view of the strict local convexity  of $\underline{u}$ and \eqref{eq3-17}.

We will need the following barrier function
\[ \Psi = A v + B \rho^2 \]
with
\[ v = u - \underline{u} + \epsilon\, d - \frac{N}{2} \,d^2 \]
and the linearized operator associated with equation \eqref{eq2-13}
\begin{equation} \label{eq2-21}
 L  = G^{ij}\, \nabla'_{i j} + G^i \,\nabla'_i
\end{equation}
to estimate the mixed tangential normal and pure normal second derivatives at $z_0$.
By direct calculation and Corollary \ref{GsGu} we have
\begin{equation} \label{eq3-20}
\begin{aligned}
 L v = \,& ( G^{ij} \nabla'_{ij} + G^i \nabla'_i  ) ( u - \underline{u} + \epsilon \,d - \frac{N}{2} d^2) \\
     = \, & G^{ij} \nabla'_{ij} ( u - \underline{u} - \frac{N}{2} \,d^2 ) + \epsilon \,G^{ij} \nabla'_{ij}  d + G^i  \nabla'_i ( u - \underline{u} + \epsilon\, d - \frac{N}{2} \,d^2 ) \\
     \leq \, & G^{ij} \Big( \nabla'_{ij} u - \,\big( \nabla'_{ij}(\underline{u} + \frac{N}{2} \,d^2 ) - 2 c_0 \delta_{ij} \big) \Big) \\ & - 2 c_0 \sum G^{ii}
       + C \epsilon \sum G^{ii} + C  ( 1 +  \epsilon +  N \delta )
\end{aligned}
\end{equation}
Since $G(\nabla'^2 u, \nabla' u, u)$ is concave  with respect to $\nabla'^2 u$,
\begin{equation} \label{eq3-21}
\begin{aligned}
& G^{ij} \Big( \nabla'_{ij} u -\, \big( \nabla'_{ij} (\underline{u} + \frac{N}{2} \,d^2 ) - 2 c_0 \delta_{ij}\big) \Big) \\ \leq &\,\, G( \nabla'^2 u, \nabla' u, u ) - G\Big( \nabla'^2\big( \underline{u} + \frac{N}{2} \,d^2 \big) - 2 c_0 I, \nabla' u, u \Big)
\end{aligned}
\end{equation}
Note that
\begin{equation*}
\begin{aligned}
&  \nabla'^2\big( \underline{u} + \frac{N}{2} \,d^2 \big ) - 2 c_0 I + u I \\
= \quad &  \nabla'^2 \underline{u} + \underline{u}\, I +  N d \nabla'^2 d + N \nabla' d \otimes \nabla' d - 2 c_0 I + ( u - \underline{u} ) I \\
\geq \quad & 2 c_0 I - C N \delta I + N \nabla' d \otimes \nabla' d := \mathcal{H}
\end{aligned}
\end{equation*}
We thus have
\begin{equation} \label{eq3-22}
\begin{aligned}
& G\Big( \nabla'^2\big( \underline{u} + \frac{N}{2} \,d^2 \big) - 2 c_0 I, \nabla' u, u \Big)  \\
= \,\, & F \Big( \frac{- \phi \,\zeta'(u)}{\sqrt{\phi^2 + \zeta'^2 |\nabla' u|^2 }}\, g^{-1/2} \big( \nabla'^2( \underline{u} + \frac{N}{2} \,d^2 ) - 2 c_0 I + u I \big)\,\, g^{-1/2} \Big) \\
\geq  &\,\, F \Big( \frac{- \phi\, \zeta'(u)}{\sqrt{\phi^2 + \zeta'^2 |\nabla' u|^2 }} \,g^{-1/2} \mathcal{H}\, g^{-1/2} \Big)
\\ =  & \,\,F \Big( \frac{- \phi\, \zeta'(u)}{\sqrt{\phi^2 + \zeta'^2 |\nabla' u|^2 }} \,\mathcal{H}^{1/2} \,g^{-1} \,\mathcal{H}^{1/2} \Big) \\
\geq & \,\,F \Big( \frac{- \phi\, \zeta'(u)}{\sqrt{\phi^2 + \zeta'^2 |\nabla' u|^2 }}\,\mathcal{H}^{1/2} \, \frac{1}{ \phi^2 + \zeta'^2(u) |\nabla' u|^2} \,I \,\mathcal{H}^{1/2} \Big) \\
=  & \,\,F \Big( \frac{- \phi\, \zeta'(u)}{(\phi^2 + \zeta'^2 |\nabla' u|^2 )^{3/2}} \,\mathcal{H} \Big)\,
\geq \, \, F ( \tilde{c}\, \mathcal{H} )
\end{aligned}
\end{equation}
where $\tilde{c}$ is a positive constant depending only on $C_0$ and $C_1$.

Combining \eqref{eq3-20}--\eqref{eq3-22} we have
\begin{equation} \label{eq3-23}
 L v \leq   \, - F ( \tilde{c}\,\mathcal{H} )   + (   C \epsilon  - 2 c_0 ) \,\sum G^{ii} + C ( 1 + \epsilon + N \delta )
\end{equation}
where $\mathcal{H} = \mbox{diag} \Big( 2 c_0 - C N \delta,\,\, \ldots,\,\, 2 c_0 - C N \delta,\,\, 2 c_0 - C N \delta + N\Big)$.
By \eqref{eq1-11}, we can choose $N$ sufficiently large and $\epsilon$, $\delta$ sufficiently small with $\delta$ depending on $N$ such that
\[  C \epsilon  \leq c_0, \quad  C N \delta \leq c_0, \quad  -F ( \tilde{c}\,\mathcal{H} )   + C  + 2 c_0  \leq - 1, \]
Therefore, \eqref{eq3-23} becomes
\begin{equation} \label{eq3-25}
 L v \leq - c_0 \sum G^{ii} - 1
\end{equation}
We then choose $\delta \leq \frac{2 \epsilon}{N}$ such that
\[ v \,\geq \, 0 \quad \mbox{in} \quad \Omega_\delta \]
A direct consequence of \eqref{eq3-25} is
\begin{equation} \label{eq3-29}
L \Psi \, = \, A \,L v \,+\, B\,L( \rho^2 ) \,\leq \,A ( - c_0 \sum G^{ii} - 1 ) + B C ( 1 + \sum G^{ii} ) \quad\mbox{in}\,\,\Omega_{\delta}
\end{equation}
which will be used later.  Besides, we also need to estimate $L ( \nabla'_k u )$. For this, we first apply the formula
\begin{equation*}
\nabla'_{ij} ( \nabla'_k u ) = \,\nabla'_k \nabla'_{ij} u + \Gamma^{l}_{ik} \nabla'_{jl} u + \Gamma_{jk}^l \nabla'_{il} u + \nabla'_{k} \Gamma_{ij}^l \,u_l
\end{equation*}
to obtain
\begin{equation} \label{eq3-30}
\begin{aligned}
L ( \nabla'_k u ) =\,& G^{ij} \nabla'_{ij} ( \nabla'_k u ) + G^i \, \nabla'_i (\nabla'_k u) \\
=  \,& \big( G^{ij}\nabla'_k \nabla'_{ij} u +  G^i \, \nabla'_{k i} u \big) \, + G^{ij} \Gamma^{l}_{ik} \nabla'_{jl} u + G^{ij} \Gamma_{jk}^l \nabla'_{il} u \\ & + G^{ij} \nabla'_{k} \Gamma_{ij}^l \,u_l   +  G^i \, \Gamma_{ik}^l \,u_l
\end{aligned}
\end{equation}
By \eqref{eq2-14} and \eqref{eq3-18} we have
\begin{equation*}
G^{ij} \Gamma^{l}_{ik} ( \nabla'_{jl} u + \,u \,\delta_{jl}) =\, F^{st} \gamma^{is} \gamma^{jt} \Gamma^{l}_{ik}  \cdot  \gamma_{jp}\, a_{pq}\,\gamma_{ql}\\
=\,( \gamma^{is} \Gamma^{l}_{ik} \gamma_{q l } ) \,F^{st} \,a_{t q}
\end{equation*}
The term $G^{ij} \Gamma_{jk}^l \nabla'_{il} u$ can be computed similarly. Taking the covariant derivative of \eqref{eq2-13} and applying Corollary \ref{GsGu} we have
\[ | G^{ij}\nabla'_k \nabla'_{ij} u +  G^i \, \nabla'_{k i} u | \leq C + | (\psi_u - G_u) u_k | \leq C (1 + \sum G^{ii}) \]
From all these above, \eqref{eq3-30} can be estimated as
\begin{equation} \label{eq3-19}
\vert L ( \nabla'_k u ) \vert \leq  \, C ( 1 + \sum G^{ii} )
\end{equation}

For fixed $\alpha < n$, choose $B$ sufficiently large such that
\[ \Psi \pm \nabla'_\alpha (u - \varphi) \geq 0 \quad \mbox{on} \quad \partial\Omega_\delta   \]
From  \eqref{eq3-29} and \eqref{eq3-19}
\[ L ( \Psi \pm \nabla'_\alpha (u - \varphi) ) \leq A ( - c_0 \sum G^{ii} - 1 ) + B C ( 1 + \sum G^{ii} )  \]
Then choose $A$ sufficiently large such that
\[ L ( \Psi \pm \nabla'_\alpha (u - \varphi) ) \leq 0 \quad \mbox{in} \quad \Omega_\delta \]
Applying the maximum principle we have
\[ \Psi \pm \nabla'_\alpha (u - \varphi) \geq 0  \quad \mbox{in} \quad \Omega_\delta \]
which implies
\begin{equation} \label{eq3-28}
\vert \nabla'_{\alpha n} u (z_0) \vert \leq C
\end{equation}

It remains to estimate the double normal derivative $\nabla'_{n n} u $ on $\partial\Omega$. By the strict local convexity of $u$, we only need to give an upper bound
\[ \nabla'_{n n} u  \leq C \quad \mbox{on} \quad \partial\Omega \]
The following proof is inspired by an idea of Trudinger \cite{Tru}. First we prove that
\begin{equation} \label{eq13}
 M := \min\limits_{z \in \partial\Omega}\,\min\limits_{\xi \in T_z( \partial \Omega ), |\xi| = 1} ( \nabla'_{\xi\xi} u + u ) \geq c_1
\end{equation}
for some constant $c_1 > 0$. Assume that $M$ is achieved at $z_1 \in \partial \Omega$ in the direction of $\xi_1$.  Let $e_1, \ldots, e_n$ be the local orthonormal frame field around $z_1$ on $\Omega \subset \mathbb{S}^n$ such that $e_1 (z_1) = \xi_1$. Thus we have
\begin{equation*}
\begin{aligned}
M = & \nabla'_{\xi_1 \xi_1} u (z_1) + u(z_1) =  \nabla'_{11} u (z_1) + u(z_1) \\
= & (\nabla'_{11} \underline{u} (z_1) + \underline{u}(z_1)) - ( u - \underline{u} )_n (z_1) \,\Gamma_{11}^n (z_1)
\end{aligned}
\end{equation*}
Assume $(u - \underline{u})_n (z_1) \,\Gamma_{11}^n (z_1) > \frac{1}{2} ( \nabla'_{11} \underline{u} (z_1) + \underline{u} (z_1))$, for, otherwise we are done. Since the function $\Gamma_{11}^n$ is continuous and $0 < ( u - \underline{u} )_n (z_1) \leq C$,
\[ \Gamma_{11}^n (z) \geq \frac{1}{2} \,\Gamma_{11}^n (z_1) \geq c_2 > 0 \quad \quad\mbox{on} \quad
 \Omega_\delta = \{ z \in \Omega | \,\mbox{dist}_{\mathbb{S}^n}(z_1, z) < \delta  \}  \]
for some small $\delta > 0$.
Now consider
\[ \Phi = \frac{\nabla'_{11} \varphi + \varphi - M }{\Gamma_{11}^n} - (u - \varphi)_n \]
Since $\nabla'_{11}(u - \varphi) = - (u - \varphi)_n\, \Gamma_{11}^n$ on $\partial\Omega$,
\[ \nabla'_{11} \varphi + \varphi - (u - \varphi)_n\, \Gamma_{11}^n = \nabla'_{11} u + u \geq M \]
Thus, $\Phi \geq 0$ on $\partial\Omega \cap \Omega_\delta$.
Also, by \eqref{eq3-19}
\begin{equation}\label{eq3-31}
L(\Phi) = \, L \Big( \frac{\nabla'_{11} \varphi + \varphi - M }{\Gamma_{11}^n} + \varphi_n \Big) - L u_n
\leq  \, C ( 1 + \sum G^{ii} )
\end{equation}
Now choose $B$ sufficiently large such that $\Psi + \Phi \geq 0$ on $\partial \Omega_\delta$. By \eqref{eq3-29} and \eqref{eq3-31} we then choose $A$ sufficiently large such that
$L(\Psi + \Phi) \leq 0$ in $\Omega_{\delta}$. Since $(\Psi + \Phi)(z_1) = 0$, we have $(\Psi + \Phi)_n (z_1) \geq 0$ and consequently
\[ u_{nn} (z_1) \leq C. \]
Together with \eqref{eq3-27} and \eqref{eq3-28}, a bound $|\nabla'^2 u (z_1)| \leq C$ can be obtained and hence a bound for all principle curvatures at $z_1$ by \eqref{eq3-18}. Therefore, the principle curvatures at $z_1$ admit a uniform positive lower bound by \eqref{eq1-7}, which in turn yields a positive lower bound for the eigenvalues of $\nabla'^2 u (z_1) + u(z_1) I $.  Hence \eqref{eq13} is proved.

By \eqref{eq13} and Lemma 1.2 in \cite{CNSIII} there exists a constant $R > 0$ depending on the estimates \eqref{eq3-27} and \eqref{eq3-28} such that if $u_{nn}(z_0) \geq R$ and $z_0 \in \partial \Omega$, then the eigenvalues $(\lambda_1, \ldots, \lambda_n)$ of $\nabla'^2 u (z_0) + u(z_0) I $ satisfy
\[ \frac{c_1}{2} \leq \lambda_\alpha \leq C, \quad \alpha = 1, \ldots, n - 1, \quad \lambda_n \geq \frac{R}{2} \]
Consequently
\[ \nabla'^2 u (z_0) + u(z_0) I \geq X^{-1} \Lambda X  \]
where $X$ is an orthogonal matrix and  $\Lambda = \mbox{diag}\Big( \frac{c_1}{2},\, \ldots,\, \frac{c_1}{2},\, \frac{R}{2} \Big)$.
Hence at $z_0$,
\begin{equation*}
\begin{aligned}
& G( \nabla'^2 u, \nabla' u, u ) (z_0) \\
= \,\, &F \Big(\frac{- \phi \zeta'(u)}{\sqrt{\phi^2 + \zeta'^2 |\nabla' u|^2 }} \, g^{-1/2} \big( \nabla'^2 u (z_0) + u(z_0) I \big)\, g^{-1/2} \Big) \\
\geq \,\, &F \Big(\frac{- \phi \zeta'(u)}{\sqrt{\phi^2 + \zeta'^2 |\nabla' u|^2 }} \, g^{-1/2} X^{-1} \Lambda X\, g^{-1/2} \Big) \\
=  \,\,& F \Big(\frac{- \phi \zeta'(u)}{\sqrt{\phi^2 + \zeta'^2 |\nabla' u|^2 }} \, \Lambda^{1/2} \,X \, g^{-1} \,X^{-1}\,\Lambda^{1/2} \Big) \\
\geq \, \,& F \Big( \frac{- \phi \zeta'(u)}{\sqrt{\phi^2 + \zeta'^2 |\nabla' u|^2 }} \, \Lambda^{1/2} \,X  \, \frac{1}{ \phi^2 + \zeta'^2(u) |\nabla' u|^2} \,I \,X^{-1}\,\Lambda^{1/2} \Big) \\
=   \,\,& F \Big( \frac{- \phi \zeta'(u)}{(\phi^2 + \zeta'^2 |\nabla' u|^2 )^{3/2}} \,\Lambda \Big) \geq \,F ( \tilde{c} \,\Lambda )
\end{aligned}
\end{equation*}
By \eqref{eq1-11} we can choose $R$ sufficiently large such that
$F ( \tilde{c} \,\Lambda ) > \sup_{\bar{\Omega} \times [C_0^{-1}, C_0]} \psi$. It follows that
\[ G( \nabla'^2 u, \nabla' u, u ) (z_0) > \psi(z_0, u(z_0))\]
which is a contradiction to equation \eqref{eq2-13}. Hence $\nabla'_{nn} u \leq R$ on $\partial\Omega$ and \eqref{eq3-32} is proved.

\vspace{5mm}

\section{Global curvature estimates}

\vspace{3mm}

The ideas for deriving global $C^2$ a priori estimates for starshaped compact or convex hypersurfaces can be found in \cite{JL05, SX15, GRW15} (see also \cite{CNSV} for vertical graphs). For strictly locally convex hypersurfaces, we synthesize the ideas in \cite{JL05, SX15} to estimate from above for the largest principal curvature $\kappa_{\max} = \max_{1 \leq i \leq n} \kappa_i$ of $\Sigma$, which, together with \eqref{eq5-1}, \eqref{eq3-32} and \eqref{eq3-5} implies an estimate for $\Vert \rho \Vert_{C^2(\overline{\Omega})}$.

\begin{thm} \label{Theorem 1}
Assume \eqref{eq1-5}, \eqref{eq1-6} and \eqref{eq1-10}. Let $ \Sigma = \{ ( z, \rho (z) ) |\, z \in \Omega \subset \mathbb{S}^n \} \subset N^{n + 1} (K)$ be a strictly locally convex $C^4$ hypersurface satisfying equation
\eqref{eq1-0} for some positive $C^2$ function $\psi$ defined on $N^{n + 1} (K)$. Suppose in addition that
\begin{equation*}
 0 < C_0^{-1} \leq \rho(z) \leq C_0 < \rho_U^K \quad \mbox{and}  \quad |\nabla' \rho| \leq C_1 \quad \mbox{on} \quad \overline{\Omega}
\end{equation*}
where $C_0$ and $C_1$ are uniform positive constants. Then there exists a constant $C$ depending only on $C_0$, $C_1$, $\Vert\psi\Vert_{C^2}$ and $\inf\psi$ such that
\[ \max\limits_{\substack{ z \in \Omega \\  i = 1, \ldots, n}}  \kappa_i (z)  \leq C \,( 1 +  \max\limits_{\substack{z \in \partial \Omega\\ i = 1, \ldots, n}}  \kappa_i (z) )  \]
\end{thm}

\begin{proof}
Since $\kappa_i > 0$ for all $i$ on $\Sigma$, it suffices to estimate from above for the largest principal curvature $\kappa_{\max}$ of $\Sigma$. To construct a test function, we will make use of the following ingredients:
\[  \Phi(\rho) = \int_0^{\rho} \phi(r) \,d r \]
and the support function
\[ \tau = \bar{g}( V,\, \nu ) = \langle V,\, \nu \rangle  \]
Note that $\tau$ has a positive lower bound
\begin{equation*}
\tau =  \,\Big\langle \phi(\rho) \frac{\partial}{\partial \rho}, \frac{- \nabla' \rho + \phi^2 \,\frac{\partial}{\partial \rho}}{\sqrt{\phi^4 + \phi^2 |\nabla' \rho|^2}} \Big\rangle  \geq \, 2 a > 0
\end{equation*}
Now define the test function
\begin{equation*}
\Theta = \ln \kappa_{\max} - \,\ln ( \tau - a ) + \beta\, \Phi
\end{equation*}
Assume $\Theta$ achieves its maximum value at ${\bf x}_0 = (z_0, \rho(z_0))\in \Sigma$.
Choose a local orthonormal frame $E_1, \ldots, E_n$ around ${\bf x}_0$ on $\Sigma$ such that $h_{ij}({\bf x}_0) = \kappa_i \,\delta_{ij}$, where $\kappa_1, \ldots, \kappa_n$ are the principal curvatures of $\Sigma$ at ${\bf x}_0$.  We may assume $\kappa_1 = \kappa_{\max} ({\bf x}_0) \geq 1$. Then, at ${\bf x}_0$,
\begin{equation} \label{eq2-3}
\frac{h_{11i}}{h_{11}}  -  \frac{{\tau}_i}{\tau - a} + \beta \,\Phi_i = 0
\end{equation}

\begin{equation} \label{eq2-4}
\frac{h_{11ii}}{h_{11}} - \frac{h_{11i}^2}{h_{11}^2} \,- \frac{{\tau}_{ii}}{\tau - a} + \Big(\frac{{\tau}_i}{\tau - a}\Big)^2 + \, \beta \,\Phi_{ii}\, \leq 0
\end{equation}
By Codazzi equation and Gauss equation we have
\begin{equation} \label{eq2-9}
\nabla_l h_{ij} = \nabla_j h_{il}
\end{equation}
and
\begin{equation} \label{eq2-5}
 h_{iill} = h_{llii} + \kappa_l \kappa_i^2 - \kappa_l^2 \kappa_i + K (\kappa_i - \kappa_l)
\end{equation}

In the rest of this section all computations are evaluated at ${\bf x}_0$.
Under the local orthonormal frame $E_1, \ldots, E_n$, equation \eqref{eq1-0} appears as
\begin{equation} \label{eq2-32}
 F(h) = f(\lambda (h)) = \psi(V) \quad\quad\mbox{where}\quad h = \{ h_{ij} \}
\end{equation}
Covariantly differentiate \eqref{eq2-32} twice we have
\begin{equation} \label{eq2-6}
 F^{ii} h_{iil} = \phi'\, d_V\psi (E_l)
\end{equation}
and
\begin{equation} \label{eq2-7}
F^{ii} h_{ii11} + F^{ij,\, kl} h_{ij1} h_{kl1} \geq - C \kappa_1
\end{equation}
Here we have used the property of the conformal Killing field $V$,
\[ \nabla_{E_l} V = \phi' \,E_l  \]
Combining \eqref{eq2-4}, \eqref{eq2-5} and \eqref{eq2-7} we have
\begin{equation} \label{eq2-8}
\begin{aligned}
& - \frac{1}{\kappa_1} F^{ij, kl} h_{ij1} h_{kl1} - \,\frac{1}{\kappa_1^2} \sum f_i  h_{11i}^2 - \sum f_i \kappa_i^2 \, + ( \kappa_1 - \frac{K}{\kappa_1} ) \sum f_i \kappa_i \\
& \,+ K \, \sum f_{i} - C \,- \frac{1}{\tau - a} \sum F^{ii} {\tau}_{ii} + \,\frac{1}{(\tau - a)^2} \sum f_i {\tau}_i^2 \,+ \,\beta
  \sum F^{ii} \Phi_{ii}  \leq 0
\end{aligned}
\end{equation}
Now we partition $\{1, \ldots, n\}$ into two parts,
\[ I = \{ j: \,f_j \leq 2 f_1 \}, \quad \quad  J = \{ j: \,f_j >  2 f_1 \}\]
By \eqref{eq2-3}, for any $\epsilon > 0$,
\begin{equation} \label{eq2-33}
 \frac{1}{\kappa_1^2} \sum\limits_{i\in I} f_i  h_{11i}^2 \,\leq  \, C (1 + \epsilon^{-1}) \beta^2 f_1 \sum \Phi_i^2  + \frac{(1 + \epsilon)}{(\tau - a)^2} \sum f_i {\tau}_i^2
\end{equation}
Taking \eqref{eq2-33}, \eqref{eq2-6} and the following equations (see Lemma 2.2 and Lemma 2.6 in \cite{GL13} for the proof)
\[ \Phi_i = \phi(\rho) \,\rho_i, \quad\quad \Phi_{ii} = \phi' - \tau \,\kappa_i\]
\[ {\tau}_i = \phi(\rho)\, \rho_i\, \kappa_i \]
\[ {\tau}_{ii} = \phi(\rho)\,\sum\limits_m \rho_m\, h_{iim} + \phi'(\rho)\,\kappa_i - \tau \,\kappa_i^2 \]
into \eqref{eq2-8} yields
\begin{equation} \label{eq2-10}
\begin{aligned}
&   - \frac{1}{\kappa_1} F^{ij, kl} h_{ij1} h_{kl1} - \,\frac{1}{\kappa_1^2} \sum\limits_{i \in J} f_i \, h_{11i}^2 \\  & + \,\big( \frac{a}{\tau - a} - \frac{C \epsilon}{( \tau - a )^2} \big) \sum f_{i} \kappa_i^2 \,- C \beta^2 (1 + \epsilon^{-1})\, f_1 \\
& + ( \kappa_1 - \frac{K}{\kappa_1} - \tau \beta - \frac{\phi'}{\tau - a} ) \sum f_i \kappa_i \,+ ( \beta \phi' + K ) \sum f_{i} \,- C - \frac{C}{a}  \leq 0
\end{aligned}
\end{equation}
Using an inequality due to Andrews \cite{And} and Gerhardt \cite{Ger} and applying \eqref{eq2-9}
\begin{equation*}
- \frac{1}{\kappa_1} F^{ij, kl} h_{ij1} h_{kl1} \geq  \frac{1}{\kappa_1} \sum\limits_{i \neq j} \frac{f_i - f_j}{\kappa_j - \kappa_i} h_{ij1}^2
\geq \frac{2}{\kappa_1} \sum\limits_{i \geq 2} \frac{f_i - f_1}{\kappa_1 - \kappa_i} h_{11i}^2 \geq \frac{1}{ \kappa_1^2} \sum\limits_{i \in J} f_i \, h_{11i}^2
\end{equation*}
where the fractions are interpreted as limits whenever the denominators are zero. Inserting it into \eqref{eq2-10}, applying assumption \eqref{eq1-10}, choosing $\epsilon = {a^2}/{(2 C)}$ and $\beta = u_L^K$ we obtain
\begin{equation*}
\frac{a^2}{2 (\tau - a)^2}  \sum f_{i} \kappa_i^2 \, - C \beta^2 ( 1 + \frac{1}{a^2} ) \, f_1 \,
+ ( \kappa_1 - 1 - \tau \beta - \frac{\phi'}{\tau - a} ) \,\sigma_0 \,- C - \frac{C}{a} \leq 0
\end{equation*}
Since $\sum f_{i} \kappa_i^2  \geq f_1 \kappa_1^2 $, a uniform upper bound for $\kappa_1$ follows easily from the above inequality. Consequently, we obtain a uniform upper bound for $\kappa_{\max}$ on $\Sigma$.
\end{proof}

\vspace{5mm}

\section{Existence in $\mathbb{R}^{n+1}$ and $\mathbb{H}^{n+1}$}

\vspace{3mm}

In this section, we will use classical continuity method and degree theory developed by Y. Y. Li \cite{Li89} to prove the existence of solution to the Dirichlet problem \eqref{eq2-26}--\eqref{eq2-25}. Under the transformation $\overline{\rho} = \zeta ( \underline{u} )$ and $\underline{u} = \eta ( \underline{v} )$, the subsolution condition \eqref{eq1-4} can be expressed as
\begin{equation} \label{eq2-29}
\left\{ \begin{aligned}
\mathcal{G}(\nabla'^2 \underline{v}, \nabla' \underline{v}, \underline{v}) \geq & \,\psi(z, \underline{v})  \quad  \mbox{in} \quad \Omega\\
\underline{v} = & \,\varphi  \quad \quad \mbox{on} \quad \partial \Omega \end{aligned} \right.
\end{equation}
Assume that $\underline{v}$ is not a solution of \eqref{eq2-26}, for otherwise we are done.
We consider the following two auxiliary equations.

\begin{equation} \label{eq6-8}
\left\{ \begin{aligned} \mathcal{G} (\nabla'^2 v, \nabla' v, v) \, =  & \, \big( t \epsilon + ( 1 - t ) \frac{\underline{\psi} (z)}{ \xi (\underline{v}) } \big)  \,\xi (v) \quad\quad & \mbox{in} \quad \Omega \\ v \,  = & \, \underline{v} \quad \quad & \mbox{on} \quad \partial\Omega \end{aligned} \right.
\end{equation}
and
\begin{equation} \label{eq6-9}
\left\{ \begin{aligned} \mathcal{G} (\nabla'^2 v, \nabla' v, v) \, =  & \,  t \, \psi(z, v) + ( 1 - t ) \,\epsilon \,\xi (v) \quad\quad & \mbox{in} \quad \Omega \\ v \,  = & \, \underline{v} \quad \quad & \mbox{on} \quad \partial\Omega \end{aligned} \right.
\end{equation}
where $t \in [0, 1]$, $\underline{\psi} (z) = \mathcal{G}( \nabla'^2 \underline{v}, \nabla' \underline{v}, \underline{v} ) (z)$, $\epsilon$ is a small positive constant such that
\begin{equation} \label{eq6-20}
\underline{\psi} (z) \,>  \epsilon \,\xi(\underline{v}) \quad\mbox{in}\quad \Omega
\end{equation}
and $\xi(v) = e^{2v}$ if $K = 0$ while $\xi(v) = \sinh v$ if $K = - 1$. The existence results in $\mathbb{R}^{n+1}$ was given in \cite{Su16} where the author assumed the existence of a strict subsolution. In this section, we will consider the case when $K = 0$ or $K = -1$ with the subsolution assumption.

\begin{lemma} \label{Lemma6-1}
Let $\psi(z)$ be a positive function defined on $\Omega$. For $z \in \Omega$ and a strictly locally convex function $v$ near $z$, if
\[\mathcal{G}(\nabla'^2 v, \nabla' v, v) (z) =  F(a_{ij}[v])(z) = f (\kappa [v])(z) = \psi(z) \,\xi(v)(z) \]
then
\[ \mathcal{G}_v (\nabla'^2 v, \nabla' v, v) (z) - \,\psi(z) \,\xi'(v)(z)  < 0\]
\end{lemma}
\begin{proof}
The proof can be found in \cite{Sui1}.
\end{proof}

\begin{lemma}  \label{Lemma6-2}
For any fixed $t \in [0, 1]$, if $\underline{V}$ and $v$ are strictly locally convex subsolution and solution to \eqref{eq6-8}, then $v \geq \underline{V}$.  Thus the Dirichlet problem \eqref{eq6-8} has at most one strictly locally convex solution.
\end{lemma}
\begin{proof}
The proof is similar to Lemma 3.40 in \cite{Sui2}.
\end{proof}

\begin{lemma} \label{Lemma6-3}
Let $v$ be a strictly locally convex solution of \eqref{eq6-9}. If $v \geq \underline{v}$ in $\Omega$, then
$v > \underline{v}$ in $\Omega$ and ${\bf n}(v - \underline{v}) > 0$ on $\partial\Omega$, where ${\bf n}$ is the interior unit normal to $\partial\Omega$.
\end{lemma}

\begin{proof}
By \eqref{eq2-29} and \eqref{eq6-20} we know that $\underline{v}$ is a strict subsolution of \eqref{eq6-9} when $t \in [0, 1)$, while it is a subsolution but not a solution of \eqref{eq6-9} when $t = 1$. It is relatively easy to prove the conclusion when $t \in [0, 1)$, following the ideas in \cite{Su16}. For the case $t = 1$:
\begin{equation*}
\left\{ \begin{aligned} \mathcal{G} (\nabla'^2 v, \nabla' v, v) \, =  & \,   \psi(z, v) \quad & \mbox{in} \quad \Omega \\ v \,  = & \, \underline{v} \quad \quad & \mbox{on} \quad \partial\Omega \end{aligned} \right.
\end{equation*}
we will make use of the maximum principle which was originally discovered in \cite{Serrin}, while more precisely stated for our purposes in \cite{GNN} (see section 1.3, p. 212). Because the maximum principle and Hopf lemma there are designed for domains in Euclidean spaces, we need to rewrite the above equation in a local coordinate system of $\mathbb{S}^n$. For convenience, we first transform the above equation back under the transformation \eqref{eq6-2} into a form as \eqref{eq2-13}:
\begin{equation} \label{eq5-7}
\left\{ \begin{aligned}
G(\nabla'^2 u, \nabla' u, u) = & \,\psi(z, u)  \quad \quad  &\mbox{in} \quad & \Omega \\
u = & \,\underline{u}  \quad \quad & \mbox{on} \quad & \partial\Omega \end{aligned} \right.
\end{equation}
Recall that $G(\nabla'^2 u, \nabla' u, u) = F(A[u])$ where $A[u] = \{ \gamma^{ik} h_{kl} \gamma^{lj} \}$. Since at this time we do not use local orthonormal frame on $\mathbb{S}^n$, but rather a local coordinate system of $\mathbb{S}^n$, $\gamma^{ik}$ and $h_{kl}$ will appear differently (comparing with \eqref{eq3-10} and \eqref{eq3-13}). Also,
condition \eqref{eq2-29} (i.e. \eqref{eq1-4}) can be rewritten as
\begin{equation*}
\left\{ \begin{aligned}
G(\nabla'^2 \underline{u}, \nabla' \underline{u}, \,\underline{u}) \geq & \,\,\psi(z, \underline{u})  \quad  &\mbox{in} \quad \Omega\\
\underline{u} = & \,\,\varphi  \quad \quad &\mbox{on} \quad \partial \Omega \end{aligned} \right.
\end{equation*}
Note that $\underline{u}$ is not a solution of \eqref{eq5-7}.

(i) We first show that if a strictly locally convex solution $u$ of \eqref{eq5-7} satisfies $u \geq \underline{u}$ in $\Omega$, then $u > \underline{u}$ in $\Omega$. Let $N \notin \Omega$ be the north pole of $\mathbb{S}^n$. Take the radial projection of $\mathbb{S}^n \setminus \{N\}$ onto $\mathbb{R}^n \times \{ -1 \} \subset \mathbb{R}^{n+1}$ and let $\tilde{\Omega}$ be the image of $\Omega$. We thus have a coordinate system $(x_1, \ldots, x_n)$ on $\mathbb{R}^n \times \{ -1 \} \cong \mathbb{R}^n$.  The metric on $\mathbb{S}^n$, its inverse, and the Christoffel symbols are given respectively by
\[ \sigma_{ij} = \frac{16}{\mu^2} \,\delta_{ij}, \quad \mu = 4 + \sum x_i^2, \quad \sigma^{ij} = \frac{\mu^2}{16}\, \delta_{ij} \]
\[ {\Gamma}_{ij}^k = - \frac{2}{\mu} \,(\delta_{ik} x_j + \delta_{jk} x_i - \delta_{ij} x_k) \]
Consequently, the metric on $\Sigma$, its inverse and the second fundamental form on $\Sigma$ are given respectively by (c.f. \cite{SX15})
\begin{equation*}
g_{ij} = \phi^2 \,\sigma_{ij} + \zeta'^2(u)\, u_i u_j
\end{equation*}
\begin{equation*}
g^{ij} = \frac{1}{\phi^2} \Big( \sigma^{ij} - \frac{\zeta'^2(u) u^i u^j}{\phi^2 + \zeta'^2(u) |\nabla' u|^2} \Big), \quad u^i = \sigma^{ik} u_k
\end{equation*}
\begin{equation*}
h_{ij} =  \frac{- \zeta'(u) \phi}{\sqrt{\phi^2 + \zeta'^2(u) |\nabla' u|^2}} ( \nabla'_{ij} u + u \,\sigma_{ij} )
\end{equation*}
The entries of the symmetric matrices $\{\gamma_{ik}\}$ and $\{\gamma^{ik}\}$ depend only on $x_1, \ldots, x_n$, $u$ and the first derivatives of $u$.

Now, setting $\tilde{u} = \mu u$ and by straightforward calculation we have
\begin{equation} \label{eq5-10}
\nabla'_{ij} u + u \,\sigma_{ij} = \frac{1}{\mu} \tilde{u}_{ij} + \frac{2 \delta_{ij}}{\mu^2} ( \tilde{u} - \sum\limits_k x_k \tilde{u}_k )
\end{equation}
and \eqref{eq5-7} can be transformed into the following form:
\begin{equation*}
\left\{ \begin{aligned}
\tilde{G}(D^2\tilde{u}, D \tilde{u}, \tilde{u}, x_1, \ldots, x_n)  = F\Big( A\Big[ \frac{\tilde{u}}{\mu} \Big]\Big) = & \,\tilde{\psi}(x_1, \ldots, x_n, \tilde{u})  \quad \quad  &\mbox{in} \quad & \tilde{\Omega} \\
\tilde{u} = & \,\mu \,\underline{u}  \quad \quad & \mbox{on} \quad & \partial\tilde{\Omega} \end{aligned} \right.
\end{equation*}
where $\tilde{u}_i = \frac{\partial\tilde{u}}{\partial x_i}$, $D \tilde{u} = ( \tilde{u}_1, \ldots, \tilde{u}_n )$, $\tilde{u}_{ij} = \frac{\partial^2 \tilde{u}}{\partial x_i \partial x_j}$ and $D^2 \tilde{u} = \{ \tilde{u}_{ij} \}$.

In view of \eqref{eq5-10} and \eqref{eq2-14} we know that
\[ \frac{\partial \tilde{G}}{\partial \tilde{u}_{ij}} = \frac{\partial F}{\partial a_{kl}} \,\frac{\partial a_{kl}}{\partial \tilde{u}_{ij}} = \frac{1}{\mu} \, \frac{\partial G}{\partial u_{ij}}  \]
Also, the function $\tilde{\underline{u}} = \mu \underline{u}$ satisfies
\begin{equation*}
\left\{ \begin{aligned}
\tilde{G}(D^2 \tilde{\underline{u}}, D \tilde{\underline{u}}, \tilde{\underline{u}}, x_1, \ldots, x_n) \geq & \,\tilde{\psi}(x_1, \ldots, x_n, \tilde{\underline{u}})  \quad  \mbox{in} \quad \tilde{\Omega}\\
\tilde{\underline{u}} = & \,\mu\underline{u}  \quad \quad \mbox{on} \quad \partial \tilde{\Omega} \end{aligned} \right.
\end{equation*}
Hence we can apply the Maximum Principle (see p. 212 of \cite{GNN}) to conclude that $\tilde{u} > \tilde{\underline{u}}$ in $\tilde{\Omega}$, which immediately yields $u > \underline{u}$ in $\Omega$.

(ii) To prove ${\bf n}(u - \underline{u}) > 0$ on $\partial\Omega$, we pick an arbitrary point $z_0 \in \partial\Omega$ and assume $z_0$ to be the north pole of $\mathbb{S}^n \subset \mathbb{R}^{n + 1}$. We introduce a local coordinate system about $z_0$ by taking the radial projection of the upper hemisphere onto the tangent hyperplane of $\mathbb{S}^n$ at $z_0$ and identifying this hyperplane to $\mathbb{R}^n$. Denote the coordinates by $(y_1, \ldots, y_n)$  and assume that the positive $y_n$-axis is the interior normal direction to $\partial\Omega \subset \mathbb{S}^n$ at $z_0$. In this coordinate system, the metric on $\mathbb{S}^n$, its inverse, and the Christoffel symbols are given respectively by (see \cite{Oliker, GS93})
\[ \sigma_{ij} = \frac{1}{\mu^2} \Big( \delta_{ij} - \frac{y_i y_j}{\mu^2} \Big), \quad \mu = \sqrt{1 + \sum y_i^2}  \]
\[ \sigma^{ij} = \mu^2 ( \delta_{ij} + y_i y_j ) \]
\[ {\Gamma}_{ij}^k = - \frac{\delta_{ik} y_j + \delta_{jk} y_i}{\mu^2} \]
The metric $g_{ij}$, its inverse $g^{ij}$ and the second fundamental form $h_{ij}$ on $\Sigma$ have the form as above.
The entries of the symmetric matrices $\{\gamma_{ik}\}$ and $\{\gamma^{ik}\}$  depend only on $y_1, \ldots, y_n$, $u$ and the first derivatives of $u$.

Now set $\tilde{u} = \mu u$. By straightforward calculation we have
\begin{equation} \label{eq5-8}
\nabla'_{ij} u + u \,\sigma_{ij} = \mu^{- 1} \tilde{u}_{ij}
\end{equation}
and equation \eqref{eq5-7} can be transformed into an equation defined in an open neighborhood of $0$ on $\mathbb{R}^n$, which is the radial projection of a neighborhood of $z_0$ on $\mathbb{S}^n$:
\begin{equation*}
\tilde{G}(D^2\tilde{u}, D \tilde{u}, \tilde{u}, y_1, \ldots, y_n)  = F\Big( A\Big[ \frac{\tilde{u}}{\mu} \Big]\Big) =  \,\tilde{\psi}(y_1, \ldots, y_n, \tilde{u})
\end{equation*}
where $\tilde{u}_i = \frac{\partial\tilde{u}}{\partial y_i}$, $D \tilde{u} = ( \tilde{u}_1, \ldots, \tilde{u}_n )$, $\tilde{u}_{ij} = \frac{\partial^2 \tilde{u}}{\partial y_i \partial y_j}$ and $D^2 \tilde{u} = \{ \tilde{u}_{ij} \}$.
In view of \eqref{eq5-8} and \eqref{eq2-14} we know that
\[ \frac{\partial \tilde{G}}{\partial \tilde{u}_{ij}} = \frac{\partial F}{\partial a_{kl}} \,\frac{\partial a_{kl}}{\partial \tilde{u}_{ij}} = \frac{1}{\mu} \, \frac{\partial G}{\partial u_{ij}}  \]
Applying Lemma H (see p. 212 of \cite{GNN}) we find that $(\tilde{u} - \tilde{\underline{u}})_n (0) > 0$ and equivalently ${\bf n}(u - \underline{u}) (z_0) > 0$.
\end{proof}

\begin{thm} \label{Theorem6-1}
For any $t \in [0, 1]$, the Dirichlet problem \eqref{eq6-8} has a unique strictly locally convex solution.
\end{thm}
\begin{proof}
The proof is the same as in \cite{Sui1}.
\end{proof}

\begin{thm} \label{Theorem6-2}
For any $t \in [0, 1]$, the Dirichlet problem \eqref{eq6-9} has a strictly locally convex solution. In particular, \eqref{eq2-26}--\eqref{eq2-25} has a strictly locally convex solution.
\end{thm}

\begin{proof}
The proof is the same with \cite{Sui1} except slight modifications.
\end{proof}

\vspace{5mm}

\section{Existence in \, $\mathbb{S}^{n + 1}_{+}$}

\vspace{5mm}

For any $\epsilon > 0$,  we want to prove the existence of a strictly locally convex solution to the Dirichlet problem when $K = 1$,
\begin{equation} \label{eq7-2}
\left\{
\begin{aligned}
G[u] := G( \nabla'^2 u, \nabla' u, u ) \,\,= & \,\, \psi(z, u) - \epsilon  \quad\quad &\mbox{in} \quad \Omega \\
u \,\, = & \,\, \varphi   &\mbox{on} \quad \partial\Omega
\end{aligned} \right.
\end{equation}
Then a strictly locally convex solution to  \eqref{eq2-13}--\eqref{eq2-12} follows from the uniform ($\epsilon$-independent) $C^2$ estimates (established in Section 3 and 4) and approximation.

As we have seen from last section, there does not exist an auxiliary equation in $\mathbb{S}^{n+1}_+$ with an invertible linearized operator. Hence we want to build a continuity process starting from an auxiliary equation in $\mathbb{R}^{n+1}$.

For this, we first consider a continuous version of \eqref{eq3-18}. For $t \in [0, 1]$, denote
\[a^t_{ij} = \frac{- (\zeta^t)' \phi^t}{\sqrt{(\phi^t)^2 + (\zeta^t)'^2 |\nabla' u|^2}} (\gamma^t)^{ik}\, ( \nabla'_{kl} u + u \,\delta_{kl} )\,(\gamma^t)^{lj}\]
where
\[(\gamma^t)^{ik} =   \,\frac{1}{\phi^t}\, \Big( \delta_{ik} - \frac{ (\zeta^t)'^2(u) u_i u_k }{\sqrt{(\phi^t)^2 + (\zeta^t)'^2(u) |\nabla' u|^2 } ( \phi^t + \sqrt{(\phi^t)^2 + (\zeta^t)'^2(u) |\nabla' u|^2 } )}\Big)\]
and
\[ \phi^t(\rho) = \frac{\sin (t \rho)}{t}, \quad \quad  \zeta^t(u) = \frac{1}{t} \, \mbox{arccot} \frac{u}{t} \]
Note that these geometric quantities on $\Sigma$ correspond to the background metric
\[ \bar{g}^t = d \rho^2 + (\phi^t)^2(\rho)\, \sigma \]
which provides a deformation process from $\mathbb{R}^{n+1}$ ($t = 0$) to $\mathbb{S}^{n+1}_+$ ($t = 1$).
Define
\[ G^t [ u ] \, = \, G^t( \nabla'^2 u, \nabla' u, u ) \, = \,F( a^t_{ij} ) \]
Hence $G^1 = G$.
The following property is true by direct calculation.
\begin{prop}
$G^t [u]$ is increasing with respect to $t$.
\end{prop}
\begin{proof}
\[ a_{ij}^t = \,\Big(1 + \frac{|\nabla' u|^2}{u^2 + t^2} \Big)^{- \frac{1}{2}}\, \,\tilde{\gamma}^{ik}\,\big( \nabla'_{kl} u \,+\, u \,\delta_{kl} \big)\,\tilde{\gamma}^{lj} \]
where
\[ \tilde{\gamma}^{ik} \,=\, \delta_{ik} - \frac{u_i \,u_k}{\sqrt{u^2 + t^2 + |\nabla' u|^2} \,\big( \sqrt{u^2 + t^2} + \sqrt{u^2 + t^2 + |\nabla' u|^2} \big)} \]

\[ \begin{aligned}
\,\frac{\partial}{\partial t} \,G^t[ u ]  = \,& \Big( 1 + \frac{|\nabla' u|^2}{u^2 + t^2} \Big)^{- 3/2}F^{ij}  \cdot \\ &  \Big(  \frac{t |\nabla' u|^2}{(u^2 + t^2)^2}  \tilde{\gamma}^{ik} +  2 \Big( 1 + \frac{|\nabla' u|^2}{u^2 + t^2} \Big) \frac{\partial \tilde{\gamma}^{ik}}{\partial t} \Big) \big( \nabla'_{kl} u + u \delta_{kl} \big) \tilde{\gamma}^{lj}  \end{aligned} \]
The inverse $(\tilde{\gamma}_{ik})$ of $(\tilde{\gamma}^{ik})$ is given by
\[ \tilde{\gamma}_{ik} \,=\, \delta_{ik} + \frac{u_i \,u_k}{ u^2 + t^2 + \sqrt{ (u^2 + t^2)^2 + (u^2 + t^2) |\nabla' u|^2} }  \]
Since
\[  \frac{\partial \tilde{\gamma}^{ik}}{\partial t} \,=\, -  \tilde{\gamma}^{ip}\,\frac{\partial \tilde{\gamma}_{pq}}{\partial t} \, \tilde{\gamma}^{qk} \]
\[ \frac{\partial \tilde{\gamma}_{pq}}{\partial t} = \, - \frac{u_p\, u_q \,t}{(u^2 + t^2)^{3/2} \sqrt{u^2 + t^2 + |\nabla' u|^2}}\]
and
\[ \tilde{\gamma}^{ik} \,u_k \,=\, \frac{\sqrt{u^2 + t^2}}{\sqrt{u^2 + t^2 + |\nabla' u|^2}} \, u_i \]
therefore,
\[ \begin{aligned} \,\frac{\partial}{\partial t}\, G^t[ u ] = & \frac{t}{\sqrt{u^2 + t^2} \big( u^2 + t^2 + |\nabla' u|^2 \big)^{3/2} }    F^{ij} \cdot \\ & \Big( |\nabla' u|^2 \delta_{iq} + 2 u_i u_q \Big) \tilde{\gamma}^{qk}\, \big( \nabla'_{kl} u + u \delta_{kl} \big)\, \tilde{\gamma}^{lj} \geq 0
\end{aligned} \]
\end{proof}

Recall that we have assumed a strictly locally convex subsolution.
\[\left\{
\begin{aligned}
G[\underline{u}] \,\,\geq & \,\,\psi(z, \underline{u})  \quad\quad &\mbox{in} \quad \Omega \\
\underline{u} \,\, = & \,\, \varphi   &\mbox{on} \quad \partial\Omega
\end{aligned} \right.
\]
Choose $\epsilon$ small such that
\[ \epsilon < \min\{   \min\limits_{\overline{\Omega}} G^0[\underline{u}], \,\, \min\limits_{\overline{\Omega}}\psi(z, \underline{u}) \} \]
By continuity, for $t \in [1 - \delta_1, \,1]$ where $\delta_1$ is a sufficiently small positive constant depending on $\epsilon$, we have
\begin{equation} \label{eq7-3}
\left\{
\begin{aligned}
G^t[\underline{u}] \,\,> & \,\, \psi(z, \underline{u}) - \,\frac{\epsilon}{2}  \quad\quad &\mbox{in} \quad \Omega \\
\underline{u} \,\, = & \,\, \varphi   &\mbox{on} \quad \partial\Omega
\end{aligned} \right.
\end{equation}

Denote $\mathcal{G}^t [\, v\, ] \,:= \mathcal{G}^t( \nabla'^2 v, \nabla' v, v )\, =: \, G^t [\, e^v \,]$. Consider the continuity process,
\begin{equation} \label{eq7-1}
\left\{ \begin{aligned}
\mathcal{G}^t [\, v\, ] \,  =  & \, \, \big( 1 - T(t) \big) \,\delta_2 \, e^{2 v} \, + \, T(t) \,\big( \psi(z, e^v) - \epsilon \big) \quad\quad & \mbox{in} \quad \Omega \\ v \,  = & \, \,\ln \varphi \quad \quad & \mbox{on} \quad \partial\Omega \end{aligned} \right.
\end{equation}
where $\delta_2$ is a small positive constant such that
\[ \delta_2  \,\max\limits_{\overline{\Omega}} \underline{u}^2 \,<\, \frac{\epsilon}{2} \]
and $T(t)$ is a smooth strictly increasing function with $T(0) = 0$,  $T(1) = 1$ satisfying
\[  \min\limits_{\overline{\Omega}} \,G^0 [\underline{u}] \, > \,2 \,\,T(1 - \delta_1)\, \,\max\limits_{ \overline{\Omega}} \psi(z, \underline{u})  \]

\begin{prop}
$\underline{v} = \ln \underline{u}$ is a strict subsolution of \eqref{eq7-1} for any $t\in [0, 1]$.
\end{prop}
\begin{proof}
For $t \in [1 - \delta_1,  1]$,
\[ \begin{aligned}
\mathcal{G}^t [ \underline{v} ]  \,= &\, G^t [\underline{u}]  \, >   \, \psi(z, \underline{u}) - \frac{\epsilon}{2} \,> \, \delta_2  \, \underline{u}^2 +   \, \big( \psi(z, \underline{u}) - \epsilon \big) \\
\,\geq  & \, \big( 1 - T(t) \big) \, \delta_2  \, e^{2 \underline{v}} \, + \,  T(t)\,\, \big( \psi(z, e^{\underline{v}}) - \epsilon \big)
\end{aligned} \]
For $t \in [0, 1 - \delta_1]$,
\[ \begin{aligned}
\mathcal{G}^t [ \underline{v} ]  \,= & \, G^t [ \underline{u} ] \, \geq  \, G^0 [\underline{u}] \, > \, \frac{\epsilon}{2} \, +  \,T(1 - \delta_1) \, \psi(z, \underline{u}) \\ \geq &
 \,  \big( 1 - T(t) \big) \, \delta_2  \, \underline{u}^2 \, + \,  T(t)\,\, \big( \psi(z, \underline{u}) - \epsilon \big) \\ = &\, \big( 1 - T(t) \big) \, \delta_2  \, e^{2 \underline{v}} \, + \,  T(t)\,\, \big( \psi(z, e^{\underline{v}}) - \epsilon \big)
\end{aligned}
\]
\end{proof}

Now we can obtain the existence results in $\mathbb{S}^{n+1}_+$.

\begin{thm} \label{Theorem7-1}
For any $t \in [0, 1]$, the Dirichlet problem \eqref{eq7-1} has a strictly locally convex solution. In particular, \eqref{eq7-2} has a strictly locally convex solution when $K = 1$.
\end{thm}

\begin{proof}
The $C^{2, \alpha}$ estimates for strictly locally convex solutions $v$ of \eqref{eq7-1} with $v \geq \underline{v}$ is equivalent to the $C^{2, \alpha}$ estimates for strictly locally convex solutions to the Dirichlet problem
\[
\left\{ \begin{aligned}
G^t [\, u\, ] \,  =  & \, \, \big( 1 - T(t) \big) \,\delta_2 \, u^2 \, + \, T(t) \,\big( \psi(z, u) - \epsilon \big) \quad\quad & \mbox{in} \quad \Omega \\ u \,  = & \, \,\varphi \quad \quad & \mbox{on} \quad \partial\Omega \end{aligned} \right.
\]
This can be established by changing $\phi$ and $\zeta$ into $\phi^t$ and $\zeta^t$ in the previous proof.  Then $C^{4, \alpha}$ estimates follows by classical Schauder theory. Thus we have the $t$-independent uniform estimates,
\begin{equation} \label{eq7-4}
\Vert v \Vert_{C^{4,\alpha}(\overline{\Omega})} < C_4 \quad \quad \mbox{and} \quad \quad
C_2^{-1}\, I <   \,\{ v_{ij} \,+\,v_i \, v_j\, + \,\delta_{ij} \} \,  < C_2\, I    \quad \mbox{in} \quad \overline{\Omega}
\end{equation}
Consider the subspace of $C^{4,\alpha}( \overline{\Omega} )$ given by
\[ C_0^{ 4, \alpha} (\overline{\Omega}) := \{ w \in C^{ 4, \alpha}( \overline{\Omega} ) \,| \,w = 0 \,\, \mbox{on} \,\, \partial\Omega \} \]
and the bounded open subset
\[ \mathcal{O} := \left\{ w \in C_0^{4, \alpha} (\overline{\Omega}) \,\left\vert\,\begin{footnotesize}\begin{aligned} & w > 0 \,\,\mbox{in}\,\,\Omega, \quad\quad \nabla'_{\bf n}\, w > 0 \,\,\mbox{on}\,\, \partial\Omega,\\ & C_2^{-1} \,I <  \, \{ (\underline{v} + w)_{ij} \,+\,(\underline{v} + w)_i \, (\underline{v} + w)_j\, + \,\delta_{ij} \} \,  < C_2 \, I   \,\, \mbox{in} \,\, \overline{\Omega} \\ & \Vert w {\Vert}_{C^{4,\alpha}(\overline{\Omega})} < C_4 + \Vert\underline{v}\Vert_{C^{4,\alpha}(\overline{\Omega})} \end{aligned}\end{footnotesize} \right.\right\} \]
Construct a map
$\mathcal{M}_t (w):  \,\mathcal{O} \times [ 0, 1 ] \rightarrow C^{2,\alpha}(\overline{\Omega})$,
\[ \mathcal{M}_t ( w ) = \mathcal{G}^t [ \underline{v} + w ] \, - \big( 1 - T(t) \big) \,\delta_2 \, e^{2 (\underline{v} + w )} - \, T(t) \, \big( \psi(z, \,e^{\underline{v} + w}) - \epsilon \big)\]
At $t = 0$, by Theorem \ref{Theorem6-1} for the case $K = 0$, there is a unique solution $v^0$ to
\eqref{eq7-1}.
By Lemma \ref{Lemma6-2} and Lemma \ref{Lemma6-3} we have $w^0 := v^0 - \underline{v}> 0$ in $\Omega$ and $\nabla'_{\bf n}\, w^0 > 0$ on $\partial\Omega$.  Moreover, $w^0$ satisfies \eqref{eq7-4} and thus $w^0 \in \mathcal{O}$.
Also, Lemma \ref{Lemma6-3} and  \eqref{eq7-4} implies that $\mathcal{M}_t( w ) = 0$ has no solution on $\partial\mathcal{O}$ for any $t \in [0, 1]$.
Besides, $\mathcal{M}_t$ is uniformly elliptic on $\mathcal{O}$ independent of $t$. Therefore,
$\deg (\mathcal{M}_t, \mathcal{O}, 0)$, the degree of $\mathcal{M}_t$ on $\mathcal{O}$ at $0$,
is well defined and independent of $t$.
Hence it suffices to compute $\deg (\mathcal{M}_0, \mathcal{O}, 0)$.

Note that $\mathcal{M}_0 ( w ) = 0$ has a unique solution $w^0 \in \mathcal{O}$. The Fr\'echet derivative of $\mathcal{M}_0$ with respect to $w$ at $w^0$ is a linear elliptic operator from $C^{4, \alpha}_0 (\overline{\Omega})$ to $C^{2, \alpha}(\overline{\Omega})$,
\begin{equation}
\mathcal{M}_{0, w} |_{w^0} ( h )  =  \,  (\mathcal{G}^0)^{ij}[ v^0 ] \nabla'_{ij} h  + (\mathcal{G}^0)^i [ v^0 ]  \nabla'_i h   + \big( (\mathcal{G}^0)_v [v^0]  - \,2 \,\delta_2 \, e^{2 v^0} \big) h
\end{equation}
By Lemma \ref{Lemma6-1}
\[ (\mathcal{G}^0)_v [v^0]  - \,2 \,\delta_2 \, e^{2 v^0}\, < 0 \quad\mbox{in} \quad \Omega \]
Thus $\mathcal{M}_{0,w} |_{w^0}$ is invertible. Applying the degree theory in \cite{Li89},
\[ \deg (\mathcal{M}_0, \mathcal{O}, 0) = \deg( \mathcal{M}_{0, w} |_{w^0}, B_1, 0) = \pm 1 \neq 0 \]
where $B_1$ is the unit ball in $C_0^{4,\alpha}(\overline{\Omega})$. Thus
\[ \deg(\mathcal{M}_t, \mathcal{O}, 0) \neq 0 \quad\mbox{for}\,\,\mbox{all}\,\,t \in [0, 1]\]
and this theorem is proved.
\end{proof}


\vspace{3mm}

\begin{thebibliography}{9}



\bibitem{And}
B. Andrews,{ \em Contraction of convex hypersurfaces in Euclidean space}, Calc. Var. PDE{ \bf 2} (1994), 151--171.


\bibitem{BLO}
J. L. M. Barbosa, J. H. S. Lira and V. I. Oliker, { \em A priori estimates for starshaped compact hypersurfaces with prescribed $m$th curvature function in space forms}, Nonlinear Problems of Mathematical Physics and Related Topics I, (2002), 35--52.


\bibitem{CNSI}
L. A. Caffarelli, L. Nirenberg and J. Spruck,{ \em The Dirichlet problem for  nonlinear second-order elliptic equations I. Monge-Amp\`ere equations}, Comm. Pure Applied Math., {\bf 37} (1984), 369--402.


\bibitem{CNSIII}
L. A. Caffarelli, L. Nirenberg and J. Spruck,{ \em The Dirichlet problem for  nonlinear second-order elliptic equations, III: Functions of the eigenvalues  of the Hessian}, Acta Math., {\bf 155}:3-4 (1985), 261--301.


\bibitem{CNSV}
L. A. Caffarelli, L. Nirenberg and J. Spruck,{ \em Nonlinear second-order elliptic equations V. The Dirichlet problem for Weingarten hypersurfaces}, Comm. Pure Appl. Math., {\bf 41} (1988), 41--70.



\bibitem{CLW18}
D. G. Chen, H. Z. Li and Z. Z. Wang,{ \em Starshaped compact hypersurfaces with prescribed Weingarten curvature in warped product manifolds}, Calc. Var. Parital. Dif., {\bf 57} (2018): 42.


\bibitem{Cruz}
F. F. Cruz,{ \em Radial graphs of constant curvature and prescribed boundary}, Calc. Var. Partial. Dif., {\bf 56} (2017): 83.


\bibitem{Evans}
L. C. Evans,{ \em Classical solutions of fully nonlinear, convex, second order  elliptic equations}, Comm. Pure Appl. Math., {\bf 35} (1982), 333--363.



\bibitem{Ger}
C. Gerhardt,{ \em Closed Weingarten hypersurfaces in Riemannian manifolds}, J. Differential Geom.{ \bf 43} (1996), 612--641.


\bibitem{GNN}
B. Gidas, W. M. Ni and L. Nirenberg, { \em Symmetry and related properties via the maximum principle}, Comm. Math. Phys., {\bf 68} (1979), 209--243.


\bibitem{Guan95}
B. Guan,{ \em On the existence and regularity of hypersurfaces of prescribed Gauss curvature with boundary}, Indiana U. Math. J.,  { \bf 44}  (44) (1995), 221--241.


\bibitem{Guan98}
B. Guan,{ \em The Dirchlet problem for Monge-Amp\`ere equations in non-convex domains and spacelike hypersurfaces of constant Gauss curvature}, Trans. Amer. Math. Soc.,  { \bf 350}  (1998), 4955--4971.


\bibitem{GS93}
B. Guan and J. Spruck,{ \em Boundary-value problems on $\mathbb{S}^n$ for surfaces of constant Gauss curvature}, Ann. of Math., { \bf 138} (1993), 601--624.





\bibitem{GS02}
B. Guan and J. Spruck,{ \em The existence of hypersurfaces of constant Gauss curvature with prescribed boundary}, J. Differential Geom., { \bf 62} (2002), 259--287.


\bibitem{GS04}
B. Guan and J. Spruck,{ \em Locally convex hypersurfaces of constant curvature with boundary}, Comm. Pure Appl. Math., { \bf 57} (2004), 1311--1331.


\bibitem{GL13}
P. F. Guan and J. F. Li,{ \em A mean curvature type flow in space forms}, Int. Math. Res. Notices, 2015(13).

\bibitem{GRW15}
P. F. Guan, C. Y. Ren and Z. Z. Wang,{ \em Global $C^2$-estimates for convex solutions of curvature equations}, Comm. Pure Appl. Math., {\bf 68}(8) (2015), 1287--1325.

\bibitem{JL05}
Q. N. Jin and Y. Y. Li, { \em Starshapded compact hypersurfaces with prescribed $k$th mean curvature in hyperbolic space}, Discrete and Continuous Dynamical Systems {\bf 15} (2005), 367--377.


\bibitem{Krylov}
N. V. Krylov,{ \em Boundedly nonhomogeneous elliptic and parabolic equations in a domain}, Izvestiya Rossiiskoi Akademii Nauk, Seriya Matematicheskaya {\bf 47} (1983), 75--108.

\bibitem{Li89}
Y. Y. Li,{ \em Degree theory for second order nonlinear elliptic operators and its applications}, Commun. Part. Diff. Eq., {\bf 14}(11) (1989), 1541--1578.


\bibitem{LO02}
Y. Y. Li and V. I. Oliker,{ \em Starshaped compact hypersurfaces with prescribed $m$-th mean curvature in elliptic space}, J. Partial Differential Equations, {\bf 15} (2002), 68--80.


\bibitem{Oliker}
V. I. Oliker, { \em Hypersurfaces in $\mathbb{R}^{n + 1}$  with prescribed Gaussian curvature and related equations of Monge-Amp\`ere type}, Comm. P.D.E. {\bf 9} (1984), 807--838.


\bibitem{Ro93}
H. Rosenberg,{ \em Hypersurfaces of constant curvature in space forms}, Bull. Sci. Math., {\bf 117} (1993), 211--239.


\bibitem{Serrin}
J. Serrin,{ \em A symmetry problem in potential theory},  Arch. Rat. Mech. Anal., {\bf 43} (1971), 304--318.


\bibitem{SX15}
J. Spruck and L. Xiao,{ \em A note on starshaped compact hypersurfaces with prescribed scalar curvature in space forms},  Rev. Mat. Iberoam., { \bf 33} (2017), No. 2, 547--554.



\bibitem{Su16}
C. Y. Su,{ \em Starshaped locally convex hypersurfaces with prescribed curvature and boundary}, J. Geom. Anal., { \bf 26}(3) (2016), 1730--1753.


\bibitem{Sui1}
Z. Sui,{ \em Strictly locally convex hypersurfaces with prescribed Gauss curvature and boundary in space forms}, arXiv:1805.08107.



\bibitem{Sui2}
Z. Sui,{ \em Convex hypersurfaces of prescribed curvature and boundary in hyperbolic space}, arXiv:1812.09554.


\bibitem{Tru}
N. S. Trudinger,{ \em On the Dirichlet problem for Hessian equations}, Acta Math., {\bf 175} (1995), 151--164.


\bibitem{TW02}
N. S. Trudinger and X. J. Wang,{ \em On locally convex hypersurfaces with boundary}, J. Reine Angew. Math., {\bf 551} (2002), 11--32.


\end{thebibliography}
\end{document}